\documentclass[preprint,reversenotenum]{elsarticle}

\makeatletter
\def\ps@pprintTitle{%
\let\@oddhead\@empty
\let\@evenhead\@empty
\def\@oddfoot{}%
\let\@evenfoot\@oddfoot}
\makeatother
\usepackage[leqno]{amsmath}
\usepackage{graphicx}
\usepackage{hyperref}
\usepackage{mathrsfs}
\usepackage{xcolor}
\usepackage{ amssymb }
\usepackage[misc,geometry]{ifsym} 
\usepackage{multirow}
\usepackage{array}
\usepackage{enumitem}
\usepackage{longtable}
\usepackage{tikz}
\usepackage{times}
\usepackage{fancyhdr,graphicx,amsmath,amssymb}
\usepackage[ruled,vlined]{algorithm2e}
\include{pythonlisting}

\allowdisplaybreaks

\newtheorem{theorem}{Theorem}%[section]
\newtheorem{cor}[theorem]{Corollary}
\newtheorem{lemma}[theorem]{Lemma}
\newtheorem{prop}[theorem]{Proposition}

\numberwithin{equation}{section}
\newtheorem{conj}[theorem]{Conjecture}
\newproof{proof}{Proof}

\newcommand{\abs}[1]{\left\vert#1\right\vert}

\newcommand{\Osq}{\hspace{1mm}\square\hspace{1mm}}
\begin{document}

\title{Computing Bounds on Product-Graph Pebbling Numbers}%
%\title{Specialized integer-programming methods for upper bounds on pebbling numbers of Cartesian-product graphs  }%

\author[1]{Franklin Kenter \fnref{fn1}}
\ead{kenter@usna.edu}
\author[1]{Daphne Skipper \corref{cor1}
   \fnref{fn2}}
\ead{skipper@usna.edu}
\author[2]{Dan Wilson}
\ead{dxwils3@gmail.com}

\cortext[cor1]{Corresponding author}
\fntext[fn1]{F. Kenter was partially funded by ONR grant N00014-18-W-X00709 and NSF Grant DMS-1719894.}
\fntext[fn2]{D. Skipper was partially funded by ONR grant N00014-18-W-X00709.}

\address[1]{United States Naval Academy, 572C Holloway Road, Annapolis, MD 21409}
\address[2]{CenturyLink, Monroe, Louisiana, 71203}

%\author{Franklin Kenter, Daphne Skipper, Dan Wilson}%
%
%
%
%\institute{U.S. Naval Academy, Annapolis MD 21402, USA\\
%\email{\{kenter,skipper$^\text{\Letter}$\}@usna.edu}\\
%\url{https://www.usna.edu/MathDept/}}

% ----------------------------------------------------------------

\begin{abstract}
Given a distribution of pebbles to the vertices of a graph, a pebbling move removes two pebbles from a single vertex and places a single pebble on an adjacent vertex. The pebbling number $\pi(G)$ is the smallest number
such that, for any distribution of $\pi(G)$ pebbles to the vertices of $G$ and choice of root vertex $r$ of $G$, 
there exists a sequence of pebbling moves that places a pebble on $r$.
Computing $\pi(G)$ is provably difficult, and recent methods for bounding $\pi(G)$ have proved computationally intractable, even for moderately sized graphs.

% The pebbling number of a graph is the smallest number $\pi(G)$ such that if $\pi(G)$ pebbles are distributed to the vertices of a graph in any configuration, there exists a sequence of pebbling moves that places a pebble on any choice of root vertex.

%regardless of the choice of root vertex $r$, for any distribution of $\pi(G)$ pebbles, 

% Graph pebbling (Chung, 1989) is a two-player game on a graph $G$.  Player one distributes ``pebbles'' to vertices and designates a root vertex.  Player two attempts to move a pebble to the root vertex via a sequence of pebbling moves, in which two pebbles are removed from one vertex in order to place a single pebble on an adjacent vertex. The pebbling number of a simple graph $G$  is the smallest number $\pi(G)$ such that if player one distributes $\pi(G)$ pebbles in \emph{any} configuration, player two can always win.  Computing $\pi(G)$ is provably difficult, and recent methods for bounding $\pi(G)$ have proved computationally intractable, even for moderately sized graphs. 

Graham conjectured that $\pi(G \Osq H) \leq \pi(G) \pi(H)$, where $G \Osq H$ is the Cartesian product of $G$ and $H$ (1989).
While the conjecture has been verified for specific families of graphs, in general it remains open. 
This study combines the focus of developing a computationally tractable, IP-based method for generating good bounds on $\pi(G \Osq H)$, with the goal of shedding light on Graham's conjecture.We provide computational results for a variety of Cartesian-product graphs, including some that are known to satisfy Graham's conjecture and some that are not.  Our approach leads to a sizable improvement on the best known bound for $\pi(L \Osq L)$, where $L$ is the Lemke graph, and $L\Osq L$ is among the smallest known potential counterexamples to Graham's conjecture.
\end{abstract}


\maketitle

% ----------------------------------------------------------------

\section{Introduction}

Graph pebbling, first introduced by Chung in 1989 \cite{chung1989pebbling}, can be described as a two-person game.  Given a connected graph, $G$, the adversary chooses a root vertex $r$ and an allocation of pebbles to vertices.  In a pebbling move,  player two chooses two pebbles at the same vertex, moves one to an adjacent vertex, and removes the other.  Player two wins if she finds a sequence of pebbling moves that results in a pebble at the root vertex $r$.
The { pebbling number} of graph $G$, denoted $\pi(G)$, represents the fewest number of pebbles such that, regardless of the  initial configuration and root given by the adversary, player two has a winning strategy.  

The original motivation for graph pebbling was to solve the following number-theoretic problem posed by Erd\H os and Lemke \cite{lemke1989addition}: ``For any set of $n$ integers, is there always a subset $S$ whose sum is $0\hspace{-.1 cm}\mod n$, and for which $\sum_{s \in S} \gcd(s,n) \le n$?'' %{\color{blue} (What is d?)} 
Kleitman and Lemke \cite{lemke1989addition} answered this question in the affirmative, and Chung \cite{chung1989pebbling} translated their technique into graph pebbling.  Since then, the study of graph pebbling has proliferated in its own right, inspiring many applications and variations; for an overview see \cite{gross2013handbook}. The translation of the number-theoretic problem to graph pebbling is nontrivial;  the reader is referred to \cite{elledge2005application} for details.

\vspace{12pt}

\noindent {\bf Graham's Conjecture:} This study is strongly motivated by  famous open questions in pebbling regarding the Cartesian-product (or simply, ``product'') of two graphs, 
$G \Osq H$:

\begin{conj}[Graham {\normalfont\cite{chung1989pebbling}}] Given connected graphs $G$ and $H$,
\label{conjBox}
 \[\pi(G \Osq H)\le \pi(G)\pi(H).\]
\end{conj}

\noindent Graham's conjecture has been resolved for specific families of graphs including products of paths \cite{chung1989pebbling}, products of cycles \cite{herscovici2003graham,snevily2002}, products of trees \cite{snevily2002}, and products of fan and wheel graphs \cite{feng2002pebbling}.  It was also proved for specific products in which one of the graphs has the so-called 2-pebbling property \cite{chung1989pebbling,snevily2002,wang2009graham}.  

One of the major hurdles in tackling Graham's conjecture is the lack of tractable computational tools.  Milans and Clark \cite{pi2pclark} showed that the decision problem of determining whether $\pi(G) < k$ is ${\Pi_2^P}$-complete.  Numerically verifying Graham's conjecture for specific graphs has been extremely difficult; as a result, there does not appear to be a discussion, let alone a consensus, regarding whether or not the conjecture is true.
 
A more practical intermediate goal is to improve the bounds on the pebbling numbers of product graphs in general, and in special cases.   To this end, Auspland, Hurlbert, and Kenter  \cite{AHK2017} proved that $\pi(G \Osq H) \le \pi(G)\ (\pi(H)+ \abs{V(H)})$. 
Since $\pi(H) \geq \abs{V(H)}$ (the adversary wins by placing a single pebble on each vertex in $V(H)\backslash r$), this result gets within a factor of two of Graham's conjecture: $\pi(G \Osq H)\le 2 \pi(G)\pi(H)$.  

When seeking a counterexample to Graham's conjecture, it is natural to focus on small graphs that do not possess the 2-pebbling property.  The { Lemke graph}, $L$, shown in Figure \ref{fig.L}, was the first graph of this kind to be discovered \cite{chung1989pebbling}. Since then, infinite families of examples have been constructed  \cite{wang2001pebbling}, but the Lemke graph, with $|L| = \pi(L) = 8$, is still among the smallest;  it was verified in \cite{newlemke} that every graph with seven or fewer vertices has the 2-pebbling property.  As suggested in \cite{gross2013handbook}, $L \Osq L$ is a potential counterexample to Graham's conjecture, and would be among the smallest.  Gao and Yin proved Graham's conjecture for $L \Osq K_n$, where $K_n$ is the complete graph on $n$ vertices, and $L \Osq T$, where $T$ is a tree \cite{gao2017lemke}.  However, there has been little progress on even bounding the product of two graphs where neither has the 2-pebbling property. 

\begin{figure}
%\centering  \includegraphics[width = .5\textwidth]{lemkeredo.png}
\centering
\begin{tikzpicture}
%% vertices
\draw[fill=black] (0,0) circle (3pt);
\draw[fill=black] (2,0) circle (3pt);
\draw[fill=black] (2,2) circle (3pt);
\draw[fill=black] (2,-2) circle (3pt);
\draw[fill=black] (4,1) circle (3pt);
\draw[fill=black] (4,-1) circle (3pt);
\draw[fill=black] (6,1) circle (3pt);
\draw[fill=black] (6,-1) circle (3pt);
%% vertex labels
\node at (-0.4,0) {$v_8$};
\node at (1.6,2) {$v_5$};
\node at (2,0.3) {$v_6$};
\node at (1.6,-2) {$v_7$};
\node at (4,1.4) {$v_3$};
\node at (4,-1.4) {$v_4$};
\node at (6,1.4) {$v_1$};
\node at (6,-1.4) {$v_2$};
%%% edges
\draw[thick] (0,0) -- (2,0) -- (4,1) -- (6,1) -- (6,-1) -- (4,-1) -- (2,-2) -- (0,0);
\draw[thick] (0,0) -- (2,2) -- (4,1);
\draw[thick] (4,-1) -- (2,0);
\draw[thick] (2,2) -- (4,-1) -- (0,0);
\draw[thick] (2,-2) -- (4,1) ;
\end{tikzpicture}

\caption{The Lemke Graph, $L$: a minimum-sized graph without the 2-pebbling property.}
\label{fig.L}
\end{figure}
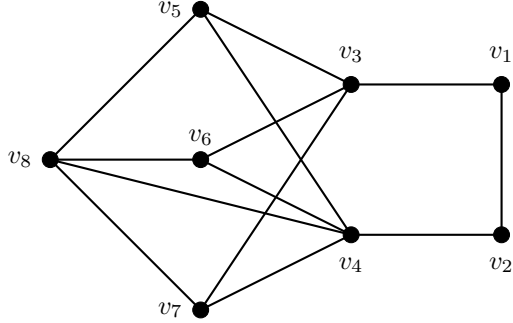
%When seeking a counterexample to Graham's conjecture, it is natural to focus on small graphs that do not possess the 2-pebbling property.  Lemke \cite{chung1989pebbling} presented a small (8-vertex) example of a graph that does not have the 2-pebbling property.  The \emph{Lemke graph}, $L$, shown in Figure \ref{fig.L} has $\pi(L) = 8$.  Since then, infinite families of graphs not possessing the 2-pebbling property have been constructed  \cite{wang2001pebbling}. Only very recently has it been verified that $L$ is among the smallest such graphs \cite{newlemke}.  Hurlbert suggested in \cite{gross2013handbook} that $L \Osq L = L \Osq L$ could be a counterexample to Graham's conjecture.  

\vspace{12pt}

\noindent {\bf Pebbling with IPs:}  Integer programs (IPs) have been applied to pebbling before, mostly in the context of product graphs.  For example, \cite{HA1998} uses an IP to reproduce $\pi(C_5 \Osq C_5)$.  
In \cite{linopt2011}, Hurlbert introduces an IP to bound $\pi(L \Osq L)$.  \cite{pip2017modified} extends Hurlbert's method to bound $|E(G)|$, where $G$ is a Class-0 graph ($\pi(G) = |G|$).

We provide a brief description of Hurlbert's contribution for context, since our work also aims at bounding $\pi(L \Osq L)$.  Hurlbert's method relies on the fact that the pebbling number of a spanning tree of a graph provides an upper bound on the pebbling number of the graph.    His model provides strong evidence that $\pi(L \Osq L) \le 108$, which is a considerable improvement on the previous bound of $\pi(L \Osq L) \leq 2(\pi(L))^2 = 128$, but is still quite far from Graham's conjectured bound of $\pi(L \Osq L) \le (\pi(L))^2 = 8^2 = 64$.  
Unfortunately, Hurlbert's technique does not scale well to $L \Osq L$, which has 64 vertices, 208 edges, and more than $10^{50}$ spanning trees.  In his full model, every  subtree corresponds to a constraint, so even writing the full model is not an option.  Still, Hurlbert makes progress by carefully selecting a subset of subtrees to translate into constraints.  

Another computational challenge is that, to fully vet a potential upper bound on $\pi(G)$, the bound must be verified for all possible roots.  Hurlbert restricted his search to root candidates that are the most likely to have large root-specific pebbling numbers, but did not verify the bound by exhausting all 64 choices of $r \in V(L \Osq L)$.

\vspace{12pt}

\noindent {\bf Partial-pebbling IP model:}  We present a novel IP approach to bounding $\pi(G \Osq H)$ that leverages the symmetry inherent in product graphs via partial-pebbling (see Section \ref{sec.partialpeb}).  In \cite{COCOApebbling}, a similar method improves the bound on $\pi(L \Osq L)$ from 108 to 91; here we extend those ideas in three significant directions and obtain $\pi(L \Osq L) \leq 85.$\footnote{These bounds have not been theoretically verified by an exact rational IP solver; rather, we use { Gurobi} which uses floating-point arithmetic.}  The first improvement, in which we we extend the partial-pebbling approach so that $G$ and $H$ play symmetric roles in the model, accounts for most of the bound improvement. 

Unlike previous IP approaches to pebbling, we directly incorporate the notion of ``2-pebbling'' (i.e., pebbling two or more pebbles at a time for a reduced cost).  Intuitively, this seems to be a critical ingredient for improving bounds on $\pi(G \Osq H)$, considering the importance of the ``2-pebbling property'' in previous results.  In fact, the inclusion of 2-pebbling greatly improved our bound on $\pi(L \Osq L)$.
In our second extension to \cite{COCOApebbling}, we generalize the modeling of 2-pebbling so that the IP applies to any pair of graphs $G$ and $H$ for which the pebbling numbers and 2-pebbling tables are known, with no special graph-specific modifications required.  

Symmetric treatment of $G$ and $H$ require modeling at the full level of granularity of $G \Osq H$, even though our constraints mostly model at the level of $G$ or $H$.  Therefore, in order to bound $\pi(G \Osq H)$, we must bound the root-specific pebbling numbers for each choice of $|V(G \Osq H)|$ root nodes.  In our third extension to \cite{COCOApebbling}, we embed the IP in an algorithm that finds the overall bound of $\pi(G \Osq H)$ without exhaustively finding the best possible bound for every choice of root node.

%The IP solver that we use for our computations, \texttt{Gurobi}  \cite{gurobi}, solves all of our models to integer optimality within a couple of seconds, so there is reason to hope that our approach may be scalable to even larger product graphs. It is also likely that our modestly-sized models can be solved by an exact rational IP solver such as \cite{exactsolver}, to theoretically validate our bounds.

In Section \ref{sec.pre}, we give an overview of pebbling and introduce partial-pebbling.  In Section \ref{sec.model}, we develop an IP model based on partial-pebbling to bound $\pi(G \Osq H)$.  In Section \ref{sec.algorithm}, we present an algorithm that uses the IP model to efficiently search all root nodes for the overall bound on the pebbling number. In Section \ref{sec:computations}, we provide computational results for a variety of product graphs, including $\pi(L \Osq L)$.  In Section \ref{sec.conclusion}, we make some closing observations and discuss future directions.

\section{Graph pebbling} \label{sec.pre}

In this section, we set the stage by introducing the graph-theoretic notation that we use, as well as concepts and notation from graph pebbling.  For a more detailed presentation of graph pebbling, see \cite{gross2013handbook}.  We also introduce partial-pebbling, which serves as the foundation for our IP model.

Throughout, we assume that our graphs are simple, undirected, and connected.   We use the notation $G := (V(G), E(G))$, to indicate a graph with vertex set $V(G)$ and edge set $E(G)$. For simplicity, we use $|G| := |V(G)|$ to denote the vertex-count of $G$, and $V(G) := \{1, 2, \dots, |G|\}$ to denote its vertex set.  Further, $i \sim_G j$ indicates that $\{i, j\} \in E(G)$, and $D_G(i,j)$ represents the graph theoretic distance between vertices $i$ and $j$ in $G$. 
%If the context is clear, we simply write $i \sim j$ or $D(i,j)$, respectively.  
The diameter of a graph is the maximum distance between any pair of its vertices. We use $\Delta_G$ to denote the maximum degree of $G$.

The Cartesian-product (also called the box-product, weak-product, or xor product) graph of $G$ and $H$, denoted $G\Osq H$ has vertex set $V(G \Osq H) := V(G) \times V(H) = \{(i,j): i \in \{1,2,\dots,|G|\}, j \in  \{1,2,\dots,|H|\}\}$.  For edges of $G \Osq H$, we have $(g, h) \sim_{G \Osq H} (g', h')$ if $g = g'$ and $h \sim_H h'$, or $h = h'$ and $g \sim_G g'$.
For example, $K_2 \Osq K_2 = C_4$, the 4-cycle, and $K_2 \Osq C_4 = Q_3$, the cube.  Although there are other common graph products, in this document every reference to a ``product'' graph refers to the Cartesian-product graph.

A natural way to conceptualize $G \Osq H$ is to think of it as the graph $H$ (which we call the frame graph), with a copy of $G$ at each vertex.  For $j \in V(H)$, $G_j$ denotes the copy of $G$ at vertex $j$, so that $V(G_j) = V(G) \times {j}$.  We say that $G_j$ is a $G$-slice of $G \Osq H$, or if the context is understood, $G_j$ is simply a slice.  Similarly, $G \Osq H$ has $H$-slices of the form $H_i$, for $i \in V(G)$.
For $i \in V(G)$ and $j_1 \neq j_2 \in H$, we have $(i,j_1) \sim_{G \Osq H} (i,j_2)$ if and only if $j_1 \sim_H j_2$.  In this case, we say that slices $G_{j_1}$ and $G_{j_2}$ are adjacent.  Also, the distance between $G_{j_1}$ and $G_{j_2}$ is $D_H(j_1,j_2)$.

\subsection{Pebbling $G$}\label{sec:pebblinG_jntro}

\newcommand{\cc}{\tilde{c}}
A configuration (or {pebbling configuration}) on $G$ is a vector of nonnegative integers $c = \left(c_1, c_2, \dots, c_{\abs{G}}\right)$, where $c_i$  represents the number of pebbles placed on vertex $i \in V(G)$.  The {support} of $c$ is the set of vertices assigned at least one pebble by $c$, $\{i \in V(G): c_i > 0\}$.  We refer to the size of $c$ as $\|c\|_1$. We refer to the support-size of $c$ as $|\{i \in V(G): c_i > 0\}|$.

A pebbling move consists of removing two pebbles from one vertex and adding one pebble to an adjacent vertex.  More generally, a $d$-hop move consists of removing $2^{d}$ pebbles from vertex $v$ and adding one pebble to vertex $w$, where $D(v,w)=d$.  We say that a configuration $c$ is solvable if, given any choice of root $r \in V(G)$, there exists a (possibly empty) sequence of pebbling moves such that the resulting configuration has at least one pebble at $r$.  Otherwise, we say that $c$ is unsolvable. The {pebbling number} of $G$, denoted $\pi(G)$, is the lowest positive integer $k$ such that all configurations of size $k$  (i.e., $\|c\|_1 = k$) are solvable.

One variant of the pebbling game is to require the second player to move two pebbles to the root in order to win.  The 2-pebbling number of $G$ with respect to support $s$, $\pi_2(G, s)$, is the minimum number of pebbles such that if a configuration on $G$ has support size $s$, 2-pebbles are guaranteed to reach the root node.  A graph $G$ has the { 2-pebbling property} if  $\pi_2(G,s) \leq 2\pi(G) - s +1$, for all $s \in \{1,2,\dots,|G|\}$. In essence, the 2-pebbling property guarantees that each additional vertex of support provides a discount of one pebble when pebbling twice.  It is worth noting that a graph, $G$,  with the 2-pebbling property may have $\pi_2(G, s) < 2\pi_G - s + 1$. For instance, for the path graph on $n$ vertices $\pi_2(P_n,n-1) = 2^{n-1} + n  \le 2 \cdots 2^{n-1} = 2 \pi(P_n)$. Since calculating $\pi_2(G, s)$ exactly is generally infeasible, we will often use $2\pi_G - s + 1$ in place of $\pi_2(G, s)$ in our models for graphs with the $2$-pebbling property, with the understanding that this may create suboptimal bounds.

The monotonic 2-pebbling number of $G$ with respect to support $s$, $\pi_2^{mon}(G,s)$, is the number of pebbles required to 2-pebble $G$ if the support size is {at least} $s$.  Often the 2-pebbling and the monotonic 2-pebbling tables are the same.  However, the 2-pebbling table of the Lemke graph is not monotonic, as shown in Table \ref{tab:2pebblingtable}.

%Capturing this value allows us to model the dynamic situation in which a 2-pebbling strategy is employed after moving pebbles into a slice (which may increase the support size on the slice). 
 
\smallskip
\begin{table}\label{tab:2pebblingtable}
\setlength{\tabcolsep}{6pt}
\renewcommand{\arraystretch}{1.2}
\[ \begin{tabular}{| l | c | c | c | c | c | c | c|  c| }
\hline
~\text{Support-size, $s$~} & ~1~ & ~2~ & ~3~ & ~4~ & ~5~ & ~6~ & ~7~ & ~8~ \\ \hline
~\text{$\pi_2(L, s)$} & 16 & 15 & 14 & 13 & 14 & 11 & 10 & 9 \\ \hline
~\text{$\pi_2^{mon}(L, s)$} & 16 & 15 & 14 & 14 & 14 & 11 & 10 & 9 \\ \hline
\end{tabular} \]
\caption{The 2-pebbling and monotonic 2-pebbling tables of the Lemke graph.}
\end{table}
\smallskip

\subsection{Partial Pebbling $G \Osq H$}\label{sec.partialpeb}

\newcommand{\bc}{\bar{c}}
 
The strength of the constraints in our IP formulation result from modeling at the level of {partial-pebblings} of product graphs.  This concept allows us to encode a  large class of similar pebbling strategies with each constraint.

When considering the product graph $G \Osq H$, we use $K$ to denote an arbitrary one of the graphs $G$ or $H$. By the symmetry inherent in the construction of product graphs, any variable, statement or constraint on $G \Osq H$ based on the graph $G$ can be adapted to deduce another variable, statement or constraint based on $H$. In some contexts, we need to reference ``the other graph'' (i.e., the graph among $G$ and $H$ that is not $K$) which we will denote $\bar K$.

A {\em partial configuration} with respect to $K$ on $G \Osq H$ allocates pebbles to the slices $K_j$, for $j \in V(\bar{K})$, rather than to individual vertices of $G \Osq H$.  In a partial configuration $\tilde{c}_K = \left(\tilde{c}_{K,1}, \tilde{c}_{K,2}, \dots, \tilde{c}_{K,\abs{\bar{K}}}\right)$, the nonnegative integer $\tilde{c}_{K,j}$ represents the number of pebbles distributed to slice $K_j$.  When a root $(r_G , r_H)$ of $G \Osq H$ is chosen, we say that $K_{r_{\bar{K}}}$ is the {$K$-root slice}.
%For clarity, we will refer to the exact root vertex as the { root-root}. 

%There is an obvious canonical map from full to partial configurations on $G \Osq H$, $\phi \colon \mathbb{Z}_{\ge 0}^{\abs{K} \abs{\bar{K}}} \to \mathbb{Z}_{\ge 0}^{\abs{\bar{K}}}$.  A partial configuration $\tilde{c}_K$ is unsolvable if there exists some $c \in 
%\phi^{-1} (\tilde{c}_K)$ that is unsolvable.  To demonstrate that $\pi(G \Osq H) \le k$ using partial configurations, one must show that every partial configuration $\tilde{c}_K$ of size $k$ is solvable.   

Normal pebbling moves cannot necessarily be made using the information of $\tilde{c}_K$ alone. For instance, if $\tilde{c}_{G,j} = |G|$, it could be that there is one pebble per vertex of $G_j$, so that no pebbling move originating in $G_j$ is possible. On the other hand, if ${c}_{G,j} > |G|$, then at least one vertex in the slice has 2 or more pebbles, and a pebbling move can be made. We say a slice is { $s$-saturated}, or has a {saturation level of $s$}, when $\tilde{c}_{K,i} \geq s|K|$.  If a slice is $(s-1)$-saturated, the pigeonhole principle guarantees that even one ``extra'' pebble (beyond the first $(s-1)|K|$)  implies the existence of a {$s$-stack}, a collection of $s$ pebbles on a single vertex.  This concept is formalized in Lemma \ref{lem:movedistance}.

We take this nuance a step further by capturing the support size of each slice.  In this case, we may assume the existence of an $s$-stack on some vertex without necessarily having $(s-1)$-saturation on the associated slice.  

Finally, we call a collection of $\pi(K)$ pebbles on the vertices of a slice $K_j$ a $K$-set, or if the context is clear, a {set}.  
The number of $K$-sets within a $K$-slice is its {set count}, and $G \Osq H$ has a {total $K$-set count}, which is the sum of the $K$-set counts over its $K$-slices.  

\section{IP model for bounding $\pi(G \Osq H)$}\label{sec.model}

\subsection{Strategy}

Let $\mathscr{U}$ be the set of all unsolvable configurations on $G \Osq H$.  We describe a relaxation $\mathscr{R}$ of $\mathscr{U}$ ($\mathscr{U} \subseteq \mathscr{R} \subseteq \mathbb{Z}^{|G \Osq H|}_{\geq 0}$), so that  
\begin{eqnarray*}
\pi(G \Osq H) &=& 1 + \max_{c \in \mathscr{U}} \left\{\|c\|_1\right\} \\
   &\leq& 1 + \max_{c \in \mathscr{R}} \left\{\|c\|_1\right\}.
\end{eqnarray*} 
So our partial-pebbling IP takes the form of $\max_{c \in \mathscr{R}} \left\{\|c\|_1\right\}$, where $\mathscr{R} \supseteq \mathscr{U}$ is the intersection of $\mathbb{Z}$ with a polytope described by linear constraints.

Each of our pebbling constraints models a successful pebbling strategy.  In other words, any partial configuration that violates a given pebbling constraint may be successfully solved via the strategy modeled by the constraint.  In this way, we know that every partial configuration not in $\mathscr{R}$ is solvable, resulting in the relaxation of $\mathscr{U}$ that we require.

Pebbling constraints rely on counting the number of pebbles $required$ at some slice or vertex (to carry out a pebbling strategy), versus the number of pebbles that are $available$ there, resulting in the standard form,
\[
\textit{available} + 1 \leq \textit{required},
\]
where $available$ and $required$ are restricted to integer values.
A partial configuration $\tilde{c}_K$ violates the constraint (and $\tilde{c}_K$ is certifiably solvable by the modeled strategy), only if $\textit{available} \geq \textit{required}$, i.e., if there are enough pebbles to carry out the strategy.   In order to maintain a relaxation of $\mathscr{U}$ (and thus a valid upper-bound on $\pi(G \Osq H)$), when exact values are not possible, we use a lower bound for $available$ and an upper bound for $required$.

%\daphne{maybe move this? get rid of it?}
%Pebbling constraints are labeled according to strategy.  For example, A.2($K,A,B$) is the second constraint in the list that models strategy A, and requires parameters $K$, $A$, and $B$.  

\subsection{Model Data}

The IP model requires the following information:

\begin{itemize}[]
\item $G$ and $H$ (as lists of nodes and edges),
\item $(r_G, r_H)$,
\item $\pi(G)$ and $\pi(H)$, and
\item 2-pebbling tables of $G$ and $H$.
\end{itemize}
From this information, we derive all of the necessary model data parameters (listed below). We make this distinction as the input data above does not directly feed into the model, whereas the derived model parameters below do.

The data parameters required for our IP model are as follows.  Everywhere that $K$ appears in the index sets and parameters listed below, it is understood to be indexed over the set $\{G,H\}$.

\bigskip

\noindent \textbf{Index Sets}

\renewcommand{\arraystretch}{1.5}
\begin{tabular}{m{0.05\textwidth} m{0.05\textwidth} p{0.75\textwidth}} 
	$V(K)$ &  $:=$& $\{1, 2, \dots, |K|\}$, vertices of $K$; \\
	$\mathcal{S}_K$ &  $:=$& $\{0, 1, \dots, \pi(\bar{K})-1\}$, possible set counts of a $K$-slice;\\
      $\mathcal{T}_K$ & $:=$ & $\{0, 1, \dots, \left\lfloor \frac{\pi(G) \pi(H) -1}{|K|} \right\rfloor\}$, possible saturation levels of a $K$-slice; \\
      $\mathcal{D}_K$ & $:=$ & $\{1, 2, \dots, diam_K\}$, possible distances between vertices in $K$;\\
      $U_{K}$ & $:=$ & $\{0\} \cup \{s \in \{1, 2, \dots, |K|\}: \pi_2(K,s) > 2\pi(K) - s + 1\},$ support sizes that correspond to non-standard 2-pebbling numbers of $K$; \\
	$U^{mon}_{K}$ & $:=$ &  $\{0\} \cup \{s \in \{1, 2, \dots, |K|\}: \pi_2^{mon}(K,s) > 2\pi(K) - s + 1\}$, support sizes that correspond to non-standard monotonic 2-pebbling numbers of $K$.
\end{tabular}

\smallskip
\noindent\textbf{Parameters}

\begin{tabular}{m{0.2\textwidth} m{0.05\textwidth} p{0.6\textwidth}} 
	$(r_G, r_H)$ &  $:=~$& the root node in $G \Osq H$;\\
      $M$ &  $:=~ $& $2\pi(G) \pi(H)$, a ``big'' constant.\\ 
      $|K|$ &  $:=~ $& vertex-count of $K$; \\
      $D_K(i,j)$ &  $:=~ $& the distance in $K$ between $i$ and $j$, for all pairs $i, j \in V(K)$;\\
	$diam_K$ & $:=~$ & the diameter of $K$;\\
	$\pi(K)$ &  $:= ~$&   pebbling number of $K$; \\
%	$\pi_2(K,s)$ & $:=~$ & the 2-pebbling table of $K$, for $s \in \{1,2,\dots,|K|\}$; \\
%	$\pi_2^{mon}(K,s)$ & $:=~$ & $\max_{t \geq s}\{\pi_2(K,t)\}$, for $s \in \{1,2,\dots,|K|\}$; \\
	$difference_{K,0}$ & $:=~$ & $-1$; \\
	$difference_{K,0}^{mon}$ & $:=~$ & $-1$; \\ 
	$difference_{K,s}$ & $:=~$ & $\pi_2(K,s) - (2\pi(K) - s + 1)$, for $s \in U_K \backslash \{0\}$; \\
	$difference_{K,s}^{mon}$ & $:=~$ & $\pi_2^{mon}(K,s) - (2\pi(K) - s + 1)$, for $s \in U_K^{mon}\backslash \{0\}.$ 
\end{tabular}

\medskip
%\daphne{clean up this section.  there are vague references.}
The index sets  $U_K$ and ``$difference$'' parameters help to model 2-pebbling dynamically in the constraints via the variables $num2peb_{K,v}$ and  $num2peb^{mon}_{K,v}$ (see below).  As a special case, we fix  $difference_{K,0} = difference_{K,0}^{mon} = -1$, which has the effect of setting the 2-pebbling number (and monotonic 2-pebbling number) of $K$ to $2\pi(K)$ when the support size is $0$. This choice ensures the correct behavior of our constraints for the case when no pebbles are assigned to a slice. 

Some constraints use the 2-pebbling numbers directly, but other constraints require the number of pebbles needed for 2-pebbling given a support of the current size or greater.  In the latter case, constraints model strategies that require pebbling into a slice and then applying a 2-pebbling discount, so we need to account for the fact that the support size could increase.  This requirement motivates the {monotonic 2-pebbling numbers}, $\pi_2^{mon}(K,s)$, as defined above.

%\daphne{We should think again about what MAX needs to be.  Should we change it to M?} 
We define $M$ as the simple upper bound on $\pi(G \Osq H)$ from \cite{AHK2017}.  Many of our constraints are enforced or relaxed based on the value of some binary variable(s).  In these constraints, $M$, or some small multiple of $M$, is used as the standard ``big $M$'' (from integer programming).

\subsection{Decision Variables}\label{sec:variables}
 
In this section, we list our decision variables, sorted numeric type, and discuss a few interesting cases. We manage the behavior of all decision variables with the linear constraints listed in Section \ref{sec:variableconstraints}.  In the table below, the index $K \in \{G,H\}$, and $j$ ranges over all $j \in V(\bar{K})$.

\medskip

\noindent\textbf{Integer variables ($\mathbb{Z}_{\geq 0}$)}
%\begin{center}
\begin{longtable}{m{0.2\textwidth} m{0.05\textwidth} p{0.6\textwidth}} 
$c_{i,j}$ & $:=$ & number of pebbles assigned to $(i,j) \in V(G \Osq H)$; \\
$\tilde{c}_{K,j}$&$ :=$ & number of pebbles assigned to $K_j$;  \\
$set_{K,j}$ &$ :=$ & $K$-set count of $K_j$, $\left\lfloor \frac{\tilde{c}_{K,j}}{\pi(K)}\right\rfloor$; \\
$extra_{K,j}$ & $:=$& number of extra pebbles on $K_j$, $\tilde{c}_{K,j} \hspace{-.1cm}\mod \pi(K)$ \\
$sat_{K,j}$ & $:=$&  saturation level of $K_j$, $\left\lfloor \frac{\tilde{c}_{K,j} - extra_{K,j}}{|K|}\right\rfloor$; \\
$pair_{K,j}$&$ :=$&  number of pairs in $extra_{K,j}$, $\left \lfloor \frac{extra_{K,j}}{2} \right\rfloor $;\\
$support_{K,j}$&$ :=$&  support size of pebbles within $K_j$; \\
$\textcolor{black}{stack_{K,j,d}}$&$ :=$&  number of $2^{d}$-stacks in $K_j$, for $d \in \mathcal{D}_{\bar{K}}$ (lower bound);\\
$n2peb_{K,j}$&$ :=$&  $\pi_2(K,support_{K,j})$; \\
$\textcolor{black}{n2peb_{K,j}^{mon}}$&$ :=$&  $\pi_2^{mon}(K,support_{K,j})$; \\ 
$\textcolor{black}{nroot_{K,j}}$&$ :=$&  number of pebbles that can reach $\bar{K}_{r_K}$  in $K_j$ (lower bound).
\end{longtable}
%\end{center}
\addtocounter{table}{-1}

\noindent\textbf{Binary variables ($0/1$)}
\begin{center}
\begin{longtable}{m{0.2\textwidth} m{0.05\textwidth} p{0.6\textwidth}} 
$\textcolor{black}{covered_{i,j}}$ & $:=$ & 1 iff $~c_{i,j} \geq 1$; \\
$\textcolor{black}{x_{K,j,t}}$& $:=$  &1 iff $~sat_{K,j} \geq t$, for $t \in \mathcal{T}_K$; \\
$y_{K,s}$&$ :=$ & 1 iff $~\sum_{j \in V(\bar{K})} set_{K,j} \geq s$, for $s \in \mathcal{S}_K$; \\
$goodStack_{K,j,d}$&$:=$ & 1 iff $\tilde{c}_{K,j} \geq (2^{\ell}-1)(support_{K,j}-1)$, for $d \in \mathcal{D}_{\bar{K}}$; \\
$can2peb_{K,j}$&$:=$ & 1 iff $~\tilde{c}_{K,j} \geq n2peb_{K,j}$; \\
$\textcolor{black}{supportIs_{K,j,s}}$ & $:=$& 1 iff $support_{K,j} = s,$ for $s \in U_K \cup U_K^{mon}$; \\
$\textcolor{black}{supportLess_{K,j,s}}$ & $:=$& 1 iff $support_{K,j} \leq s,$ for $s \in U_K \cup U_K^{mon}$; \\
$\textcolor{black}{supportMore_{K,j,s}}$ & $:=$ & 1 iff $support_{K,j} \geq s,$ for $s \in U_K \cup U_K^{mon}$.
\end{longtable}
\end{center}
\addtocounter{table}{-1}

There are a few details about the variables that are worth noting.  
\begin{itemize}
\setlength\itemsep{1 em}
\item The pebbles counted by $extra_{K,j}$ are ``extra'' in the sense that they contribute neither to the set count nor the saturation level of $K_j$.  
\item  All use of pebbling numbers and 2-pebbling numbers of subgraphs of $G \Osq H$ are still valid for use in the integer program when replaced by upper bounds on those pebbling numbers.  Of course, tighter bounds result in stronger constraints.
\item  As mentioned in the comments about the parameters, if the support size of $K_j$ is 0, we define $n2peb_{K,j} = n2peb^{mon}_{K,j} = 2\pi(K)$.  
\item The variable $nroot_{K,j}$ takes advantage of the 2-pebbling discount on the first two pebbles.  Within $K_j$, two pebbles reach $\bar{K}_{r_K}$ with the first $n2peb_{K,j}$ pebbles; beyond that, one pebble per set reaches $\bar{K}_{r_K}$ (as a lower bound).
%\item The variable $Aset_{K,j,A}$ {can} be defined for any subgraph of $K$ that contains $r_K$.  Rather than 
\item The following variables are only used to define the other variables, and are not used in any pebbling constraints directly:  $covered_{i,j}$, $goodStack_{K,j,d}$, $support_{K,j}$, $supportIs_{K,j,s}$, $supportLess_{K,j,s}$, $supportMore_{K,j,s}$.
\end{itemize}
The constraints defining the behavior of all variables are listed in Section \ref{sec:variableconstraints}.

%\begin{itemize}
%\item
%\daphne{The blue variables are different from the way they were named in the first paper.  The indices are reversed for x and stack.  nroot replaces nHr.}
%\item
%\textcolor{purple}{ Purple indicates a new variable. }
%\item
%\textcolor{gray}{ Gray indicates a variable that can go away.}
%\end{itemize}
%\bigskip

%We calculate the lower bounding $stack_{\ell,j}$ as
%\[
%  stack_{\ell,j} = \left\lceil \frac{c_j - (2^{\ell} - 1)support_j}{2^{\ell}} \right\rceil,
%\]
%because the maximum possible number of surplus pebbles (that are not part of a $2^\ell$-stack in $G_j$) is $(2^\ell - 1)support_j$. 

\subsection{Pebbling Constraints}\label{sec:pebconstraints}
Each of our pebbling constraints fits (at least roughly) into one of two categories based on the pebbling strategy it models.
In the descriptions of these strategies, as in our pebbling constraints, we are thinking of $\bar{K}$ as the frame graph, with a copy of $K$ at each node.

\begin{description}
\item[Strategy A:] Collect enough $K$-sets among the $K$-slices to ensure that $\pi(\bar{K})$ pebbles can reach copies of $r_K$ using within $K$-slice moves.  This creates a $\bar{K}$-set in the $\bar{K}$-root slice with which to pebble the root.
\item[Strategy B:] Use between slice moves in the direction of the $K$-root slice to collect $\pi(K)$ of pebbles there.  Use these pebbles to reach the root within the $K$-root slice.
\end{description}
We label each pebbling constraint as either $A$ or $B$, to indicated the basic strategy that it employs.  For example, \ref{cons:A2} is the second constraint in the list that models strategy $A$, and it is indexed over the parameters $K$, $v$, $\eta$ and $S$.

%\item[Strategy  C:] 
% a targeted attack on the root $(r_G, r_H)$ results from accumulating a $2^{\ell}$-stack on the copy of the $r_K$ in $K_j$, where $D_{\bar{K}}(j,r_{\bar{K}}) = \ell$.  

\subsubsection{Strategy A Constraints.}  The constraints in this section model the accumulation of enough pebbles in the $\bar{K}$-root slice to pebble the root within that slice.  Most model the accumulation of at least $\pi(\bar{K})$ $K$-sets among the $K$-slices per Lemma \ref{thm:catA}, so it is important to account for the total $K$-set count in the initial configuration, which is captured by the $y_{K,s}$ variables. %All of the strategy A constraints are based on Lemma \ref{thm:catA}.

\begin{lemma}\label{thm:catA}
Any configuration that has a total $K$-set count of at least $\pi(\bar{K})$ is solvable, for $K \in \{G,H\}$.
\end{lemma}
\begin{proof}
Without loss of generality, let $K := G$.  We can use a set in $G_j$ to move a pebble to any vertex of $G_j$, and in particular to $(r_G,j)$, the vertex in the intersection of $G_j$ and $H_{r_G}$.  If there are $\pi(H)$ sets across all slices, then we can move $\pi(H)$ pebbles into $H_{r_G}$ to reach the root vertex $(r_G, r_H)$ within $H_{r_G}$.\qed
\end{proof}

\begin{theorem}\label{thm:A1}
For $K \in \{G,H\}$, inequality 
\[
\displaystyle\sum_{i \in V(\bar{K})} set_{K,i} + 1 \leq \pi(\bar{K}) \label{A1}\tag{A.1($K$)}
\] 
is valid for $\mathscr{U}$.
\end{theorem}

\begin{proof}
This is a direct consequence of Lemma \ref{thm:catA}.
\qed
\end{proof}

%While constraint set \ref{A1} is a direct consequence Lemma \ref{thm:catA},  the rest of the strategy A constraints model the accumulation of $\pi(\bar{K})$ sets within the $K$-slices via pebbling moves between $K$-slices.  

The next lemma relates the number of pebbles on a slice to the distance of between-slice moves that are possible from that slice, and follows easily by the pigeonhole principle.

\begin{lemma} \label{lem:movedistance} 
Let $K \in \{G,H\}$.  If there are at least $(2^{d}-1)|K| + 1$, pebbles on slice $K_{j}$ (or equivalently, $K_{j}$ is $(2^{d}-1)$-saturated with at least one extra pebble), then it is possible to make an $d$-hop move from $K_j$.
\end{lemma} 
The lemma follows immediately from the pigeonhole principle; as there must be a vertex with $2^d$ pebbles, thereby allowing a $d$-hop move. \qed

Constraint \ref{cons:A2}, below, models the accumulation of one or more extra sets among the $K$-slices by pebbling from a central slice $K_v$, for $v \in V(\bar{K})$, to complete partial sets at nearby slices.  In particular, if $S \subseteq V(\bar{K})\backslash\{v\}$ indexes the slices where new sets are completed, extra pebbles at $K_{v}$, along with pebbles from $|S| - \eta$ sets in $K_{v}$, are used to complete $|S|$ sets at nearby slices, for a total increase of $\eta$ new sets.  This constraint is deactivated if the current total $K$-set count is less than $\pi(\bar{K}) - \eta$.  The choice of $S$ is limited by the size of $\pi(K)$.  For example, in the setting of $G = H = L$, a 3-hop move from $G_v$ would exhaust an entire set, so the nodes in $S$ are no more than a distance of 2 from $v$ in $\bar{K}$.

%Constraints \ref{cons:IA2} model the accumulation of one or more additional sets by shifting pebbles from a single $K$-slice to $\alpha_{d}$ $K$-slices a distance $d$ away, for $d \in \mathcal{D}_{\bar{K}}$.  In particular, for $v \in \bar{K}$, this pebbling strategy uses extra pebbles in $K_v$, along with pebbles from $\alpha_1 + \alpha_2 + \dots + \alpha_{diam_{\bar{K}}} - \eta$ sets in $G_v$, in order to accumulate $\alpha_1 + \alpha_2 + \dots + \alpha_{diam_{\bar{K}}}$ new sets, increasing the total set count by $\eta$.  These constraints could be extended to include pebbling moves of greater distances, but in the setting of $G = H = L$, our current focus, a 3-hop move from $G_v$ would exhaust an entire set.

\begin{theorem}\label{thm.12-shift} Fix $K \in \{G,H\}$, $v \in V(\bar{K})$, and $\eta \in \{1, 2, \dots, \pi(\bar{K})\}$.  Select $S \subseteq V(\bar{K})\backslash\{v\}$ such that 
\[ 
d ~:=~ \max_{w \in S} \{D_{\bar{K}}(v,w)\} ~\leq~ \lceil\log_2(\pi(K))\rceil-1,
\]
and $\eta \leq |S| \leq \pi(\bar{K})$.  Then
\begin{equation} \tag{A.2($K,v, \eta, S$)} \label{cons:A2}
\begin{cases}
|K|\left(|S| - \eta\right) &\hspace{-2mm}+ ~extra_{K,v} + 1 ~\leq~\\
&\displaystyle \sum_{w \in S} 2^{D_{\bar{K}}(v,w)}(\pi(K) - extra_{K,w}) \\
&\quad+ ~M\left(1 - x_{K,v,\chi}\right) \\
&\quad+ ~M\left(1 - y_{K,\pi(\bar{K}) - \eta}\right), 
\end{cases}
\end{equation}
where $\chi = (2^d - 1) + |S| - \eta$, is valid for $\mathscr{U}$.
\end{theorem}

\begin{proof} WLOG, let $K := G$.  We may assume that the total $G$-set count is at least $\pi(H) - \eta$, and that $G_v$ is at least $\chi$-saturated; otherwise, the constraint is relaxed by one of the $M$ terms.  

Due to the saturation level at $G_v$, by Lemma \ref{lem:movedistance} there are at least 
$(|S| - \eta)|G| + extra_v$ pebbles available in $G_v$ to be used in $d$-hop moves.
It costs $2^{\ell}$ pebbles to make an $\ell$-hop move, so the number of pebbles required in $G_v$ to complete one set per element of $S$, is $\sum_{w \in S} 2^{D_{\bar{K}}(v,w)}(\pi(K) - extra_{K,w})$.

If the constraint is violated, enough pebbles are available to carry out this strategy: up to $(|S| - \eta)|G|$ pebbles may used from $G_v$ in order to create $|S|$ new sets, one in each of the $G_j$, for $j \in S$.  Since $|G| \leq \pi(G)$, no more than $(|S| - \eta)$ sets are disassembled at $G_v$.  This strategy increases the total set count by at least $\eta$. \qed
\end{proof}
%\begin{proof} %We may assume that $\sum_i {set}_i \ge \pi(H) - \eta$ and vertex $i$ is at least $(3 + \alpha + \beta - \eta)$-saturated. Otherwise, the coefficient of one of the ${MAX}$ terms is at least 1, and the bound holds.  Since each$i \in A$ is 3-saturated by at least 2${pairq}_i$ pebbles, we can perform $\sum_i {pairq}_i$ pebbling moves each moving a vertex to any vertex in $B$. If the number of such moves exceeds $\sum_{j \in B}(\pi(G) - q_j)$, then these pebbling moves can be allocated in such a way to complete each of these sets. In which case the resulting configuration will have an additional $\beta$ sets. In which case, the resulting configuration has at least   $\pi(H) - \beta + \beta =  \pi(H)$ sets, violating Theorem \ref{thm:1a1}.
%\textcolor{blue} {To do.}
%\franklin{A the start of a proof is in the comments.}
%\qed\end{proof}

%\begin{figure} \label{fig:IA2}
%
%[Figure of what the above constraint is doing]
%
%\caption{A schematic for the pebbling moves in Theorem \ref{thm.12-shift}}
%
%\end{figure}

Each constraint in the next class requires a complete bipartite subgraph of $\bar{K}$ with vertex partition $S \cup T \subseteq V(\bar{K})$.  Such a constraint models the collection of one additional set at each slice $K_j$, for $j \in T$, using only extra pebbles from the slices $K_i$, for $i \in S$.  This strategy increases the total $K$-set count by $|B|$.  Note that this constraint is {\bf not symmetric} with respect to $S$ and $T$.

\begin{theorem} Fix $K \in \{G,H\}$.  Let $(S,T)$ be an ordered pair of disjoint subsets of $V(\bar{K})$, with $i \sim_{\bar{K}} j$ for all $i \in S$, $j \in T$. Then 
\begin{equation} \tag{A.3($K,S,T$)}\label{cons:A3}  
\begin{cases}
\displaystyle \sum_{i \in S} pair_{K,i} + 1 \leq &\hspace{-3mm} ~~ \displaystyle\sum_{j \in T}(\pi(K) - extra_{K,j}) \\
& \quad + M  \left(|S| - \sum_{i \in S} x_{K,i,1} \right)\\
& \quad + M  \left(1 - y_{K,\pi(\bar{K}) - |T|} \right)
\end{cases}
\end{equation}
is valid for $\mathscr{U}$.  
\end{theorem} 

%\begin{theorem} Fix positive integers $\alpha, \beta \leq \Delta_H$, with $\alpha + \beta \leq |H|$.  Suppose $A$ and $B$ are disjoint subsets of $V(H)$ of sizes $\alpha$ and $\beta$, respectively, with $i \sim_H j$ for all $i \in A$, $j \in B$.  Then 
%\begin{equation} \tag{I.A.3($\alpha, \beta$)}\label{cons:IA3}  
%\begin{cases}
%\sum_{i \in A} pair_{i} + 1 \leq &\hspace{-3mm} ~~\sum_{j \in B}(\pi(G) - extra_j) \\
%& \quad + MAX \cdot (\alpha - \sum_{i \in A} x_{1,i} )\\
%& \quad + MAX \cdot (1 - y_{(\pi(H) - \beta)})
%\end{cases}
%\end{equation}
%is valid for $\mathscr{U}$.
%\end{theorem} 

\begin{proof} WLOG, let $K := G$. If the constraint is enforced, the total $G$-set count is at least $\pi(H) - |T|$, and $G_i$ is at least 1-saturated, for each $i \in S$.    The sum $\sum_{i \in S} pair_{G,i}$ counts the number of pairs of pebbles in slices indexed by $S$ that can be used to 1-pebble to neighboring slices indexed by $T$ (while decreasing neither the saturation level nor the set count at the slices indexed by $S$).  The summation $\sum_{j \in T}(\pi(G) - extra_{G,j})$ captures the cumulative number of pebbles required at the slices indexed by $T$ to build a complete set in each.  When the constraint is violated, there are enough pebbles to increase the total $G$-set count by $|T|$. \qed
\end{proof}

Rather than collecting $\pi(\bar{K})$ sets among the $K$-slices, the strategy for constraint \ref{A4} involves counting the number of pebbles that can reach the $\bar{K}$-root slice via within $K$-slice pebbling moves.  This constraint is the first that employs a 2-pebbling discount, which is ``hidden'' in the variables $nroot_{K,j}$.  Constraint \ref{A4} is a strengthened version of constraint \ref{A1}.

\begin{theorem} For $K \in \{G,H\}$, the inequality
\begin{equation} \tag{A.4($K$)}\label{A4} 
  \sum_{j \in V(\bar{K})} nroot_{K,j} + 1 ~\leq~ \pi(\bar{K}),
\end{equation}
is valid for $\mathscr{U}$.
\end{theorem}

%\begin{theorem} The inequality
%\begin{equation} \tag{A.4}\label{A4} 
%  \sum_{j \in V(H)} nHr_j + 1 ~\leq~ \pi(H),
%\end{equation}
%is valid for $\mathscr{U}$.
%\end{theorem} 

The next set of strategy A constraints requires an $\alpha$-star subgraph of $\bar{K}$ with central vertex $v$, and comes into play when $K_v$ is highly pebbled.  It models pebbling from $K_v$ to build sets in each of the $\alpha$ neighboring $K$-slices, and finishing out a collection or $\pi(\bar{K})$ pebbles in $\bar{K}_{r_K}$ by pebbling to the copy of $r_K$ within $K_v$ with a 2-pebbling discount.

\begin{theorem} Fix $K \in \{G,H\}$, $v \in V(\bar{K})$, and $S \subseteq \{w \in V(\bar{K}): w \sim_{\bar{K}} v\}$ such that $1 \leq |S| \leq  \pi(\bar{K})-\textcolor{black}{3}$.  Then,
 
\begin{equation} \tag{A.5($K,v, S$)}\label{A5}  \begin{cases}
  \tilde{c}_{K,v} + 1 &~\leq~ \left(2 \displaystyle \sum_{j \in S} (\pi(K) - extra_{K,j})\right) \\ & \quad + n2peb_{K,v} + (\pi(\bar{K}) - (2 + |S|))\pi(K)
  \end{cases}
\end{equation}
is valid for $\mathscr{U}$.
\end{theorem}

%\begin{theorem} Fix $\alpha \leq \min\{\Delta_H, \pi(H)-2\}$ in $\mathbb{Z}_{\geq 0}$.  For any $v \in V(H)$, let $A \subset V(H)$, such that $\abs{A} = \alpha$, and $D(v,j) = 1$, for all $j \in A$.  Then 
%\begin{equation} \tag{III.A.5($\alpha$)}\label{IIIA5} 
%  \tilde{c}_v + 1 ~\leq~ \left(2 \sum_{j \in A} (\pi(G) - extra_j)\right) + n2peb_v + (\pi(H) - (2 + \alpha))\pi(G),
%\end{equation}
%is valid for $\mathscr{U}$.
%\end{theorem}
\begin{proof} WLOG, let $K := G$.
If the constraint is violated, enough pebbles are available in $G_v$ to carry out the following strategy.  First use $n2peb_{G,v}$ pebbles to move 2 pebbles to $(r_G,v)$ (in $H_{r_G}$) within $G_v$.  

The upper bound on $|S|$ ensures that $(\pi(H) - (2 + |S|))\pi(G) \geq \pi(G) \geq |G|$, so at least $2 \sum_{j \in S} (\pi(G) - extra_{G,j})$ pebbles are available for moves to adjacent slices.  Use these to build an additional set in each slice $G_j$, for $j \in A$.  Use these sets to put $|S|$ pebbles into the $H$-root slice, $H_{r_G}$.

This leaves at least $(\pi(\bar{K}) - (2 + |S|))\pi(K)$ pebbles in $K_v$.  Use these to finish out a set of $\pi(H)$ pebbles in $H_{r_G}$
\qed
\end{proof}
%\textcolor{blue}{We changed this value from 2 to 3 because $\pi(\bar{K}) - (2 + \alpha)$ must be at least 1 in order to guarantee 1 saturation to use the $2 \sum_{j \in A} (\pi(K) - extra_{K,j})$ pebbles to move out of this slice even after using the $n2peb_{K,v}$ pebbles to 2-pebble (without disturbing the support level first).}

%\subsubsection{Other Pebbling Constraints.}\label{sec.2p}

%\begin{proof}
%If the constraint is violated, $\pi(H)$ pebbles can be assembled in $H_{r_G}$. \qed
%\end{proof}

The last set of strategy $A$ constraints model a variation of strategy $A$.  Rather than accumulating a $\bar{K}$-set in the $\bar{K}$-root slice, we accumulate a $2^d$-stack on a node within the $\bar{K}$-root slice that is a distance of $d$ away from the root node.
In particular, the constraints model pebbling into slice $K_j$, and then applying a 2-pebbling discount there to build a $2^d$ stack on the copy of $r_K$ in $K_j$, where $d := D_{\bar{K}}(j,r_{\bar{K}})$.

\begin{theorem} Fix $K \in \{G,H\}$ and $v \in V(\bar{K})\backslash \{r_{\bar{K}}\}$, and let $d := D_{\bar{K}}(v,r_{\bar{K}})$.  Then
\begin{equation} \tag{A.6($K,v$)}\label{C1} 
  \left(\sum_{j \in V(\bar{K})\backslash \{r_{\bar{K}}\}} stack_{K,j,D_{\bar{K}}(v,j)}\right) + \tilde{c}_{K,v} + 1 ~\leq~\\ 
 {n2peb}_{K,v}^{mon} + (2^{d}-2)\pi(K),
\end{equation}
is valid for $\mathscr{U}$. 
\end{theorem}

\begin{proof}  WLOG, let $K:=G$.
If  $D_{H}(j,v) = \ell$, $stack_{G,j,\ell}$ counts the number of pebbles that $G_j$ can contribute to $G_v$. After using stacks to pebble into $G_v$, the total number of pebbles available at $G_v$, is the expression on the left (without the $``+1"$).
 
If the constraint is violated, enough pebbles are available at $G_v$ to carry out the following strategy.  Use the stacks to pebble into $G_v$.  Use ${n2peb}_{K,v}^{mon}$ of the pebbles now on $G_v$ to move 2 pebbles to $(r_G,v)$.  Since we pebbled into $G_v$ before 2-pebbling $G_v$, it is necessary to reserve ${n2peb}_{K,v}^{mon}$ pebbles for 2-pebbling, rather than ${n2peb}_{K,v}$ pebbles.  Next we use $(2^{d}-2)\pi(G)$ pebbles to move an additional $2^{d}-2$ pebbles to $(r_G,v)$, with no 2-pebbling discount. With $2^{d}$ pebbles at $(r_G,v)$, a $d$-hop move places one pebble at $(r_G,r_H)$. \qed
\end{proof}

\subsubsection{Strategy B Constraints.}  For strategy B, it is no longer important to track the total $K$-set count.  Instead, we use as many pebbles as possible to build a set in the $K$-root slice, in order to pebble the root node.  This means we can use stacks, rather than saturation levels, to determine the between-slice moves that are possible.  %Note that the use of stacks has a tendency to spread pebbles out within slices.

The first strategy $B$ constraint is very straight-forward.

\begin{theorem}  The following equation is valid for $\mathscr{U}$:
\begin{equation} \tag{B.1}\label{const:B1}
   set_{r_H} = 0.
\end{equation}
\end{theorem}

The next strategy B constraint, \ref{B2}, models using stacks of pebbles at all non-root slices to build a set in $K_{r_{\bar{K}}}$.  %Note that in strategy A, it is important to keep track of the number of sets, and to carefully track the increase in total set count that results from a violated A constraint.  In strategy B, however, the total set count does not matter, and we use as many pebbles as possible from each slice to build a set in the G-root slice.

\begin{theorem} Fix $K \in \{G,H\}$.  The inequality
\begin{equation} \tag{B.2($K$)}\label{B2} 
  \left(\sum_{j \in V(\bar{K})\backslash\{r_{\bar{K}}\}} stack_{K,j,D_{\bar{K}}(j,r_{\bar{K}})}\right) + \tilde{c}_{K,r_{\bar{K}}} + 1 ~\leq~ \pi(K) ,
\end{equation}
is valid for $\mathscr{U}$.
\end{theorem}

\begin{proof}  The variable $stack_{K,j,d}$ provides a lower bound on the number of $d$-hop moves that are possible from $K_j$.  If the constraint is violated, enough pebbles can reach $K_{r_{\bar{K}}}$ to complete a $K$-set.\qed
\end{proof}

%\daphne{Need to add II.B.3($\alpha$) constraints, which pebble along a path of $G$-slices to build a set in $G_{r_H}$. This constraint does not use support sizes or $k$-stacks, but I am putting it in Level II because it requires us to select a specific root vertex.  (I'm open to discussion about this going in Level I vs Level II.)  (This constraint is currently called I.A.4($\alpha = 3$) in the model file.)} 

Constraint \ref{B2} requires that each stack is used to pebble directly to the target slice, and does not allow for the collection of ``loose'' pebbles along the way. The next constraint, \ref{B3}, allows for this possibility along a path (of $K$-slices) of length $\alpha$ terminating at $K_{r_{\bar{K}}}$.  

\begin{theorem} For $K \in \{G,H\}$, let $P$ be a path in $\bar{K}$ of edge length $\alpha \in \mathcal{D}_{\bar{K}}$, with $r_{\bar{K}} = p_0 \sim_P p_1 \cdots \sim_P p_\alpha$. Then, the inequality
\begin{equation} \tag{B.3($K,P$)}\label{B3} 
\begin{cases} 
 1+ \displaystyle \sum_{i=1}^\alpha 2^{\alpha - i} (\tilde{c}_{K,p_i} - |K|)  ~\leq~& 2^\alpha (\pi(K) - \tilde{c}_{K,r_{\bar{K}}}) \\ \quad &+ 2^\alpha \cdot M \left(\alpha - \displaystyle \sum_{i=1}^{\alpha} x_{K,p_i,1}\right) 
  \end{cases}
\end{equation}
is valid for $\mathscr{U}$.
\end{theorem}

\begin{proof} 
WLOG, let $K := G$. If not all $K$-slices corresponding to the vertices of the path are $1$-saturated, then the rightmost term is positive and the constraint is relaxed. Hence, we can assume that each $K$-slice on the path is $1$-saturated.

We proceed by induction on $\alpha$, the edge-length of the path. Consider the base case, $\alpha = 1$. We will show that if the constraint is violated, there is a strategy that pebbles the root. If the constraint is violated, $K_{p_1}$ has at least $2(\pi(K) - \tilde{c}_{K,r_{\bar{K}}})+|K| - 1$ pebbles, of which $2(\pi(K) - \tilde{c}_{K,r_{\bar{K}}})$ of them can be used to add $\pi(K) - \tilde{c}_{K,r_{\bar{K}}}$ pebbles to $K_{r_{\bar K}}$. Combined with the existing $\tilde{c}_{K,r_{\bar{K}}}$ pebbles on $K_{r_{\bar K}}$, there are at least $\pi(K)$ pebbles to pebble the root.

Now, assume the constraint is valid when $\alpha = n-1$; we will show that if the constraint is violated when $\alpha = n$, there is a strategy that pebbles the root. Specifically, we will show that there are sufficient many pebbles on $K_{p_n}$ that can be moved to $K_{p_{n-1}}$ to pebble the root via the induction assumption.

 If the constraint is violated for $\alpha = n$, then
 \[ \displaystyle \sum_{i=1}^n 2^{n - i} (\tilde{c}_{K,p_i} - |K|) - 2^n (\pi(K) - \tilde{c}_{K,r_{\bar{K}}})   ~\geq~  0 . \]

Also, slice $K_{p_n}$ has at least 
\[  2^{n} (\pi(K) - \tilde{c}_{K,r_{\bar{K}}}) - \sum_{i=1}^{n-1} 2^{n - i} (\tilde{c}_{K,p_i} - |K|) + (|K| - 1) \]
pebbles.
Using the fact that that $K_{p_n}$ is 1-saturated, the pebbles on the slice $K_{p_n}$ in excess of $|K| - 1$ can be used to place half as many pebbles on the slice $K_{p_{n-1}}$. Hence, 
\[ 2^{n-1} (\pi(K) - \tilde{c}_{K,r_{\bar{K}}}) - \sum_{i=1}^{n-1} 2^{n - i - 1} (\tilde{c}_{K,p_i} - |K|) \]
pebbles can be placed onto slice $K_{p_{n-1}}$. The resulting configuration, $c'$ has 
\begin{align*}
\sum_{i=1}^{n-1} 2^{n - i} &(\tilde{c}'_{K,p_i} - |K|) - 2^n (\pi(K) - \tilde{c}'_{K,r_{\bar{K}}})  \\
&= \quad \sum_{i=1}^{n-1} 2^{n - i} (\tilde{c}_{K,p_i} - |K|) - 2^n (\pi(K) - \tilde{c}_{K,r_{\bar{K}}}) \\ &\quad\quad\quad+ 2^{n-1} (\pi(K) - \tilde{c}_{K,r_{\bar{K}}}) - \sum_{i=1}^{n-1} 2^{n - i - 1} (\tilde{c}_{K,p_i} - |K|) \\
&=\quad \sum_{i=1}^{n-1} 2^{(n-1) - i} (\tilde{c}_{K,p_i} - |K|) - 2^{n-1} (\pi(K) - \tilde{c}_{K,r_{\bar{K}}})  \\
 & \geq 0.
% && 2^{n-1} (\pi(K) - \tilde{c}_{K,r_{\bar{K}}}) - \left(\sum_{i=1}^{n-1} 2^{n - i - 1} (\tilde{c}_{K,p_i} - |K|) \right) 
\end{align*}
The first equality follows from $\tilde{c}' = \tilde{c}$ for all slices $K_{p_0},  \ldots, K_{p_{n-2}}$ except $\tilde{c}'_{p_{n-1}}$ is $\tilde{c}_{p_{n-1}}$ plus the number of pebbles moved onto it from $K_{p_n}$. The second equality follows by combining the like terms, and the final inequality follows from the assumption that the constraint is violated.

Hence, the new configuration has
\[ \sum_{i=1}^{n-1} 2^{n - i} (\tilde{c}'_{K,p_i} - |K|) - 2^n (\pi(K) - \tilde{c}'_{K,r_{\bar{K}}}) \ge 0 \]
over slices corresponding to a path of $n-1$ vertices. Using the induction hypothesis, there is a strategy to pebble the root. \qed
\end{proof}

\subsection{Variable-defining Constraints}\label{sec:variableconstraints}
For completeness, we list the constraints that manage the behavior of the variables, inserting commentary for a few interesting cases.  Refer to Section \ref{sec:variables} for descriptions of the variables, including numeric types.

\subsubsection{Partial-configuration variables}

The first two variable constraints manage the behavior of the variables that describe partial configurations with respect to each choice of frame graphs.  The full-configuration variables, $c_{i,j}$, ensure consistency between the partial configurations defined from each perspective. 

\begin{align}
\tilde{c}_{G,j} &~=~ \displaystyle\sum_{i \in V(G)} c_{i,j}, ~~&\mbox{ for } j \in V(H); \\
\tilde{c}_{H,i} &~=~ \displaystyle\sum_{j \in V(H)} c_{i,j}, ~~&\mbox{ for } i \in V(G).  
\end{align}
\medskip

\subsubsection{Sets, pairs, and saturation}
The constraints in this section are mostly straightforward.  We note only that the denominators in \ref{ylower} and \ref{xlower} are chosen so that the fraction never exceeds one.  All constraints in this section are defined for $K \in \{G,H\}$ and $j \in V(\bar{K})$.

The variables related to sets and pairs are modeled as:
\begin{align}
\tilde{c}_{K,j} &~=~ \pi(K) \cdot set_{K,j} + extra_{K,j};  \\
extra_{K,j} &~\leq~ \pi(K) - 1; \\
pair_{K,j} &~\leq~  \frac{extra_{K,j}}{2}; \\
pair_{K,j} &~\geq~ \frac{extra_{K,j} - 1}{2}; \\
  y_{K,0} &~=~ 1; \\
  y_{K,s} &~\leq~ \frac{\sum_{j \in \bar{K}} set_{K,j}}{s}, & \mbox{ for } s \in S_K \backslash \{0\}; \\
  y_{K,s} &~\geq~ \frac{\left(\sum_{j \in \bar{K}} set_{K,j}\right) - s + 1}{\pi(\bar{K})}, & \mbox{ for } s \in S_K \backslash \{0\}. \label{ylower}
\end{align}

The variables related to saturation-levels are modeled as:
\begin{align}
sat_{K,j} &~\leq~ \frac{\tilde{c}_{K,j}}{|K|}; \\
sat_{K,j} &~\geq~ \frac{\tilde{c}_{K,j} - |K| + 1}{|K|};\\
  x_{K, j, 0} &~=~ 1; \\
  x_{K, j, t} &~\leq~ \frac{sat_{K,j}}{t}, & \mbox{ for }  t \in \mathcal{T}_K \backslash \{0\}; \\
  x_{K, j, t} &~\geq~ \frac{sat_{K,j} - t + 1}{|\mathcal{T}_K|+1},  & \mbox{ for }  t \in \mathcal{T}_K \backslash \{0\}; \label{xlower}.
\end{align}

\subsubsection{Stacks}

The $support_{K,j}$ variables are required to describe the $stack_{K,j,\ell}$ variables.  
\begin{align}
covered_{i,j} &~\leq~ c_{i,j},  & \mbox{ for } (i,j) \in V(G \Osq H); \\
covered_{i,j} &~\geq~ \frac{c_{i,j}}{M}, & \mbox{ for } (i,j) \in V(G \Osq H); \\ 
support_{G,j} &~=~ \displaystyle\sum_{i \in V(G)} covered_{i,j}, & \mbox{ for } j \in V(H); \\
support_{H,i} &~=~ \displaystyle\sum_{j \in V(H)} covered_{i,j}, & \mbox{ for } i \in V(G).
\end{align}

The constraints that define the $stack_{K,j,\ell}$ variables require more care.   The logic of these constraints rests on the fact that if all of the $2^{\ell}$-stacks of pebbles are removed from the support vertices, $(2^{\ell} - 1)(support_{K,j}-1)$ is the maximum possible number of leftover pebbles.  However, the upper bound based on this logic, \ref{stackUpper}, fails if $\tilde{c}_{K,j} < (2^{\ell} - 1)(support_{K,j} -1)$, in which case $stack_{K,j,\ell}$ should be 0.  The variable $goodStack_{K,j,\ell}$ serves as a binary indicator for $\tilde{c}_{K,j} \geq (2^{\ell}-1)(support_{K,j} -1)$.  With this binary indicator, we can impose the alternate upperbound, \ref{stackUpper2}, when appropriate. The following constraints are indexed over $K \in \{G,H\}$, $j \in V(\bar{K})$, and $\ell \in \mathcal{D}_{\bar{K}}$.

\begin{align}
2^{\ell}\cdot stack_{K,j,\ell} &~\leq~ \tilde{c}_{K,j} - (2^{\ell}-1)(support_{K,j} -1) \label{stackUpper} \\
&\quad\quad + M(1 - goodStack_{K,j,\ell}); \nonumber  \\
2^{\ell}\cdot stack_{K,j,\ell} &~\leq~ M(goodStack_{K,j,\ell}); \label{stackUpper2} \\
2^{\ell}\cdot stack_{K,j,\ell} &~\geq~ \tilde{c}_{K,j} - (2^{\ell}-1)(support_{K,j}); \label{stackLower}\\
goodStack_{K,j,\ell} &~\leq~ 1 + \frac{\tilde{c}_{K,j} - (2^{\ell}-1)(support_{K,j} -1)}{2\pi(G)\pi(H)}; \\
goodStack_{K,j,\ell} &~\geq~ \frac{\tilde{c}_{K,j} - (2^{\ell}-1)(support_{K,j} -1)}
{2\pi(G)\pi(H)}.
\end{align}

\subsection{2-pebbling}
The variables $supportLess_{K,j,s}$ and $supportLess_{K,j,s}$ are used to define \\ $supportIs_{K,j,s}$, which, in turn, is required for the special case support sizes of 2-pebbling numbers and monotonic 2-pebbling numbers.  The index sets $U_K$ and $U_K^{mon}$ are indicators for these special case support sizes.

The following constraints are defined for $K \in \{G,H\}$, $j \in V(\bar{K})$, and $s \in (U_K \cup U_K^{mon}) \backslash \{0\}$.
\begin{align}
\textcolor{black}{supportLess_{K,j,s}} &\textcolor{black}{~\leq~ (|K| - support_{K,j}+ 1)/(|K| - s+1);}\\
%supportLess_{K,j,s} &~\geq~ (s - 1 - support_{K,j})/|K| \\
\textcolor{black}{supportLess_{K,j,s}} &\textcolor{black}{~\geq~ (s + 1 - support_{K,j})/|K|;} \\
%supportMore_{K,j,s} &~\leq~ (support_{K,j})/(s-1) \\
%\textcolor{blue}{supportMore_{K,j,s}} &\textcolor{blue}{~\leq~ (support_{K,j})/s} \nonumber\\
\textcolor{black}{supportMore_{K,j,s}} &\textcolor{black}{~\leq~ (support_{K,j}+1)/(s+1);} \\
%supportMore_{K,j,s} &~\geq~ (support_{K,j} - s - 2)/|K| \\
\textcolor{black}{supportMore_{K,j,s}} &\textcolor{black}{~\geq~ (support_{K,j} - s + 1)/|K|;}  \\
supportIs_{K,j,s} &~\leq~ supportMore_{K,j,s}; \\
supportIs_{K,j,s} &~\leq~ supportLess_{K,j,s}; \\
supportIs_{K,j,s} + 1 &~\geq~ supportMore_{K,j,s} + supportLess_{K,j,s}. 
\end{align}
The last 3 constraints enforce that $supportIs_{K,j,s} = 1$ if and only if\\ $supportMore_{K,j,s} = supportLess_{K,j,s} = 1$. 
% To avoid division by zero, we define $supportIs_{K,j,0}$ as special cases, due to  
%
%\begin{align}
%	supportIs_{K,j,0} &~\geq~ 1 - support_{K,j} \\
%	supportIs_{K,j,0} &~\leq~ 1 - \frac{support_{K,j}}{|K|}
%\end{align}

 Now, we are ready to define the 2-pebbling variables for $K \in \{G,H\}$ and $j \in V(\bar{K})$.  Our default values for both $n2peb_{K,j}$ and  $n2peb_{K,j}^{mon}$ match the formula for the bounds for 2-pebblable graphs:  $pi_K - support_{K,j} + 1$. The $supportIs_{K,j,s}$ variables allow us to correct the 2-pebbling numbers for any support sizes that deviate from this default value.
\begin{align}
n2peb_{K,j} &~=~ 2\pi(K) - support_{K,j} + 1 + \sum_{s \in U_K} difference_{K,s} supportIs_{K,j,s}; \\
n2peb_{K,j}^{mon} &~=~ 2\pi(K) - support_{K,j} + 1 + \sum_{s \in U_K^{mon}} difference_{K,s}^{mon} supportIs_{K,j,s}.
\end{align}

With $n2peb_{K,j}$ and $n2peb_{K,j}^{mon}$, we can define the behavior of the rest of the variables related to 2-pebbling.  These constraints are also defined for $K \in \{G,H\}$ and $j \in V(\bar{K})$.

\begin{align}
	M(can2peb_{K,j}) &~\geq~ \tilde{c}_{K,j} - n2peb_{K,j} + 1; \\
      M(1 - can2peb_{K,j}) &~\geq~ n2peb_{K,j} - \tilde{c}_{K,j};\\
	nroot_{K,j} &~\geq~ 2 (can2peb_{K,j}) + \frac{\tilde{c}_{K,j} - n2peb_{K,j} + 1}{\pi(K)} -1; \label{nroot_lbound} \\
%	nroot_{K,j} &~\geq~ \frac{\tilde{c}_{K,j} + 1}{\pi(K)} - 1 \\
	nroot_{G,j} &~\geq~ \frac{\tilde{c}_{G,j} - c_{r_G,j} + 1}{\pi(G)} - 1 + c_{r_G,j}, \quad \mbox{ for } j \in V(H); \label{nroot_lbound_G} \\
	nroot_{H,i} &~\geq~ \frac{\tilde{c}_{H,i} - c_{i,r_H} + 1}{\pi(H)} - 1 + c_{i,r_H}, \quad \mbox{ for } i \in V(G). \label{nroot_lbound_H}
\end{align}

\noindent Constraint \ref{nroot_lbound} gives credit for 2-pebbling if possible, then charges the regular 1-pebbling price for each pebble after the first two.  Bounds \ref{nroot_lbound_G} and \ref{nroot_lbound_H} give credit for any pebbles already sitting at the root node, then charges the regular 1-pebbling price for any additional pebbles to reach the root node.  The value of $nroot_{K,j}$ is determined by the maximum of the bounds applied to it.  For example, constraint \ref{nroot_lbound} is dominated by \ref{nroot_lbound_G} (or \ref{nroot_lbound_H}) whenever $\tilde{c}_{K,i} < n2peb_{K,i}$.
 
The ``$-1$'' on the right side of \ref{nroot_lbound} prevents over-counting when $(\tilde{c}_{K,i} - n2peb_{K,j})/\pi(K)$ is not a whole number.  The ``$+1$'' inside the parenthesis offsets the ``$-1$'' when $(\tilde{c}_{K,i} - n2peb_{K,j})/\pi(K)$ is a whole number.  \ref{nroot_lbound_G} and \ref{nroot_lbound_H}  behave similarly.

\section{IP-based Algorithm} \label{sec.algorithm} \label{sec:alg}
To verify a bound on $\pi(G \Osq H)$, it must be verified for every choice of root node.    We hope to avoid solving the IP to optimality for all $|G||H|$ possible root nodes because, depending on the performance of the IP, this could take a very long time.
To simplify this process, we make note of the following:
\begin{itemize}
\item For symmetric base graphs $G$, like $K_n$ or $K_{n,n}$, we only need to examine one node as root.  Without loss of generality, we choose $r_G := 1$.
\item If we have established that the maximum value for the model is $N$ for some root node, we may terminate examination of another root node when the upper bound provided by branch-and-bound (obtained from the IP solver) is less than $N$.
\item For our selection of base graphs $G$ and $H$, we label the vertices so that we expect the maximum bound to occur when the root is $(1, 1)$.  Typically, this case will need to be solved to optimality.
\end{itemize}

This leads us to Algorithm \ref{alg.algorithm1}, in which we explore the possible root nodes with the IP using a decreasing relative optimality gap in successive iterations.  Because the IP (as described in Section \ref{sec.model}) solves much more quickly for even small non-zero optimality gaps, we can quickly reduce the problem by applying the second rule above.  In Algorithm \ref{alg.algorithm1}, $OPT(r,gap)$ returns an ordered pair $(n, u)$ for root node $r$ and optimality gap $gap$.  Here, $n$ is the maximum objective value found by the solver for a feasible integer solution, and $u$ is the current upper bound on the actual optimal solution found by the solver.  We will always have $n <= u$ for any value of $gap$, with $n = u$ when $gap=0$.
%\begin{enumerate}
%\item Find the set "roots" of root nodes to investigate.
%\item Solve root node (1, 1) to optimality.  From this we establish a maximum N.
%\item Set the optimality gap to 10% for the first iteration, 5% for the second iteration, and 0 for the final iteration.
%\item For each root node in roots, if this is the first iteration for this root or the known MIP upper bound is greater than N, solve the MIP with the given optimality gap.  If the best integer solution produces and objective value M that is greater than N, set N to M.
%\item If all upper bounds are less than or equal to N, the solution to the model is N and the upper bound on pi(G x H) is N+1.
%\end{enumerate}

\begin{algorithm}[h!] \label{alg.algorithm1}
\SetAlgoLined
\KwResult{Bound on $\pi(G \Osq H)$.}
 $GAP := (0.1, 0.05, 0)$\;
 $\mathcal{S} := V(G \Osq H)$\;
 $r_0 := (1,1)$\;
 $(N, u_{r_0}) := OPT(r_0,0)$\;
 \While{$\mathcal{S} \ne \emptyset$}{
  $gap := GAP[i]$\;
  $s := \emptyset$\;
  \For{$r$ in $\mathcal{S}$}{
     \If{($u_r$ is undefined) or ($u_{r_0} > N$)}{
     $(n,u_r) := OPT(r,gap)$\;
        \If{$n > N$}{
	  $N := n$\;
        }
        \If{$u_r > N$}{
	  $s := s \cup \{r\}$\;
        }     
     }
   }{
   $\mathcal{S} := s$\;
   $i := i+1$\;
  }
 }
 \caption{IP-Based Algorithm for Bounding $\pi(G \Osq H)$}
\end{algorithm}

\section{Computations}\label{sec:computations}

\subsection{Test cases}

While this study was originally motivated by generating new upper bounds for $\pi(L \Osq L)$, our method may be applied to bound $\pi(G \Osq H)$ for general {base graphs} $G$ and $H$. %We will restrict our focus to particular pairings of graphs. 
One challenge in applying the model to a base graph $G$ is that the values of $\pi(G)$, let alone  $\pi_2(G, s)$, may not be known exactly. Hence, we restrict our computational testing to base graphs for which sufficient information is known about $\pi(G)$ and $\pi_2(G, s)$. %It is worth noting that the model only requires {\it upper bounds} for $\pi_2(G, s)$ and $\pi_G$ as the parameters within the model use $\pi_2(G, s)$ and $\pi_G$ to implement particular strategies with the requisite number of pebbles. If $\pi_2(G, s)$ and/or $\pi_G$ are overestimated, then the model simply expends more pebbles than necessary to implement these strategies. Hence, if $G$ has the 2-pebbling property, we can apply the model with $\pi_G(s)$ as $2\pi_G-s+1$ even though the true value of $\pi(G,s)$  may be lower. (For instance, for the path graph on $n$ vertices $\pi_2(P_n,n-1) = 2^{n-1} + n  \le 2 \cdots 2^{n-1} = 2 \pi(P_n)$). 
As a result, we focus on graphs that have the 2-pebbling property and whose value for $\pi_G$ is known, graphs for which Graham's conjecture is true in some cases, and/or  graphs without the two-pebbling property whose values for $\pi_2(G, s)$ are known.

In particular, we test all possible combinations of product graphs composed of the base graphs listed in Table \ref{tab.graphs}.  This results in 105 test cases, because we allow $G = H$, and because $G \Osq H = H \Osq G$.  

Our test graphs are sorted by number of vertices and by whether or not they have the 2-pebbling property.  Our ``small graphs'' have 7 or 8 vertices, since they are comparable to the Lemke graph in this respect. To better understand the computational viability and accuracy of the model as the size of the input graphs grow, we include a similar family of ``large graphs'' on 11 or 12 vertices.  For variety, our selected graphs have a wide range of pebbling numbers (from $8$ to $2^{12}$). 

\begin{table}[h!] \label{tab.basegraphs}
\centering
\begin{tabular}{ | r c | c |  c  |  c  |  c  |  }
\hline
& &$|V(G)|$ & $|E(G)|$ &$\pi(G)$ & 2PP?  \\ \hline
Original Lemke Graph & $L$ & 8 & 13 &8 & No \\
New Lemke Graph  \#1 & $L_1$ &8 & 12 &8 & No  \\
New Lemke Graph \#2  & $L_2$ & 8 &14 & 8& No \\ \hline
7-cycle & $C_{7}$ & 7& 7 & 11 & Yes \\
8-cycle &  $C_{8}$ & 8  &  8 & 16 & Yes \\
8-vertex Path & $P_{8}$ &8 &7 &$2^{8}$&Yes  \\
Complete Bipartite (4,4) & $K_{4,4}$ & 8 & 16&  12& Yes  \\
Complete (8) & $K_{8}$ &8 & 28 &8&Yes  \\
%Random Trees* & $T_{8}$  & 8&&&Yes 
\hline
11-cycle & $C_{11}$ & 11 & 11 & 43  & Yes \\
12-cycle &  $C_{12}$ & 12 & 12 & 64 & Yes \\
12-vertex Path & $P_{12}$ &12  & 11 &$2^{12}$&Yes \\
Complete Bipartite (6,6) & $K_{6,6}$ & 12 & 36 &  12& Yes  \\
Complete (12) & $K_{12}$ &12 & 66 &12&Yes \\
%Random Trees* & $T_{12}$ & 12 &&&Yes
\hline
\end{tabular}

\caption{Base graphs used to construct products $G \Osq H$ for computational testing.}
\label{tab.graphs}
\end{table}

First, we include three Lemke graphs, each on 8 vertices.  ``The'' Lemke graph, $L$, shown in Figure \ref{fig.L}, was the first example of a graph without the 2-pebbling property. Any graph without the 2-pebbling property is now called ``a'' Lemke graph. There are 22 distinct Lemke graphs on 8 vertices \cite{newlemke}, with three of them minimal as subgraphs. We focus on the three minimal cases: $L$, along with the other two minimal Lemke graphs, $L_1$ and $L_2$, as shown in Figure \ref{fig.newlemkes}. The pebbling numbers of the new Lemke graphs match the original Lemke graph, $\pi(L_1) = \pi(L_2) = \pi(L) = 8$, and the the 2-pebbling numbers match as well: $\pi_2(L_1,s) = \pi_2(L_2,s) = \pi_2(L,s)$ for $s = 1, 2, \ldots, 8$, with $\pi(L_1,5) = \pi(L_2,5) = 14$. Because of their similarity to the Lemke graph and their significance with respect to Graham's conjecture, we use them to test the adaptability of our model.

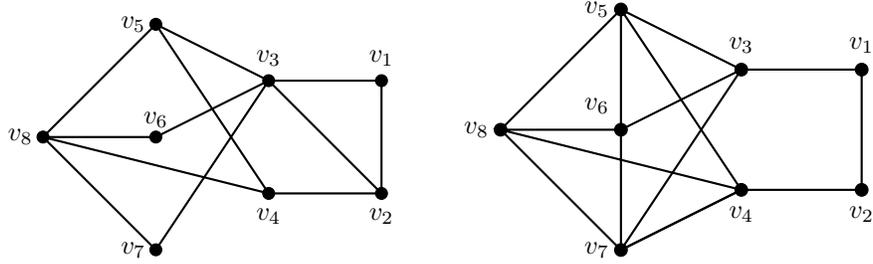
\begin{figure}
%\franklin{2x check the tikz figures one more time; I'm going crazy}
%\includegraphics[width=0.45\textwidth]{L1} \quad\quad \includegraphics[width=0.45\textwidth]{L2}
\centering
\noindent\begin{tikzpicture}[scale=0.75]
%% vertices
\draw[fill=black] (0,0) circle (3pt);
\draw[fill=black] (2,0) circle (3pt);
\draw[fill=black] (2,2) circle (3pt);
\draw[fill=black] (2,-2) circle (3pt);
\draw[fill=black] (4,1) circle (3pt);
\draw[fill=black] (4,-1) circle (3pt);
\draw[fill=black] (6,1) circle (3pt);
\draw[fill=black] (6,-1) circle (3pt);
%% vertex labels
\node at (-0.4,0) {$v_8$};
\node at (1.6,2) {$v_5$};
\node at (2,0.3) {$v_6$};
\node at (1.6,-2) {$v_7$};
\node at (4,1.4) {$v_3$};
\node at (4,-1.4) {$v_4$};
\node at (6,1.4) {$v_1$};
\node at (6,-1.4) {$v_2$};
%%% edges
\draw[thick] (0,0) -- (2,0) -- (4,1) -- (6,1) -- (6,-1) -- (4,-1); 
\draw[thick] (2,-2) -- (0,0);  
\draw[thick] (0,0) -- (2,2) -- (4,1);
\draw[thick] (2,2) -- (4,-1) -- (0,0);
\draw[thick] (2,-2) -- (4,1) ;
\draw[thick] (4,1) -- (6,-1) ;
\end{tikzpicture}\quad\quad
\begin{tikzpicture}[scale=0.8]
%% vertices
\draw[fill=black] (6,1) circle (3pt);
\draw[fill=black] (6,-1) circle (3pt);
\draw[fill=black] (4,1) circle (3pt);
\draw[fill=black] (4,-1) circle (3pt);
\draw[fill=black] (2,2) circle (3pt);
\draw[fill=black] (2,-2) circle (3pt);
\draw[fill=black] (2,0) circle (3pt);
\draw[fill=black] (0,0) circle (3pt);
%% vertex labels
\node at (1.6,0.4) {$v_6$};
\node at (-0.4,0) {$v_8$};
\node at (1.6,2) {$v_5$};
\node at (1.6,-2) {$v_7$};
\node at (4,1.4) {$v_3$};
\node at (4,-1.4) {$v_4$};
\node at (6,1.4) {$v_1$};
\node at (6,-1.4) {$v_2$};
%%% edges
\draw[thick] (0,0) -- (2,2) -- (4,1) -- (6,1) -- (6,-1) -- (4,-1) -- (2,-2) -- (0,0);
\draw[thick] (2,0) -- (2,-2) -- (4,-1) -- (2,2) -- (2,0) -- (0,0) -- (4,-1);
\draw[thick]  (2,-2) -- (4,1) -- (2,0);
\end{tikzpicture}
\caption{As shown in \cite{newlemke}, the graphs $L_1$ (left) and $L_2$ (right), along with the original Lemke graph, $L$, are all of the minimal graphs on 8 vertices without the 2-pebbling property.}
\label{fig.newlemkes}
\end{figure}

For the rest of our base graphs, we select representatives from very specific families of graphs for which the individual pebbling numbers are known and/or Graham's conjecture is known to be true. If $G$ has the 2-pebbling property, then $\pi(G \Osq H) \le \pi(G) \pi(H)$ whenever $H$ is an even cycle \cite{herscovici2003graham}, a tree \cite{gctrees}, a complete graph \cite{chung1989pebbling}, or a complete bipartite graph \cite{gcknm}. Further, the value of $\pi(G)$ for each of these graphs is known and they are known to have the 2-pebbling property themselves (see \cite{chung1989pebbling,gcknm,herscovici2003graham,gctrees}).  It is also known that odd cycles have the 2-pebbling property and that the product of two odd cycles obeys Graham's conjecture \cite{herscovici2003graham}.  Even though the Lemke graph has unusual properties when it comes to pebbling, it is known that Graham's conjecture is true when $G$ is the Lemke graph and $H$ is a complete graph or a tree \cite{gao2017lemke}.

\subsection{Computing environment}

Tests were run on a Macbook Pro (15-inch, 2017) with a 2.8GHz Intel Core i7 processor and 16 GB 2133 MHz LPDDR3 of memory.  Software included Python 3.63, Pyomo 5.61, and Gurobi 8.1.0.

\subsection{Computational results}

The results of our computations can be found in Tables \ref{tab.lemkelemke}--\ref{tab.bigbig}. There are three numbers provided for each test instance.  
\begin{itemize}
	\item The top number is the upper bound for $\pi(G \Osq H)$ attained by our model.  
	\item The middle number (in parenthesis) is $\pi(G)\pi(H)$, the
upper bound for $\pi(G \Osq H)$ suggested by Graham's conjecture.  We indicate the current state of thet conjectured bound by how it is displayed. For each instance, $\pi(G)\pi(H)$ is in {\bf bold} if Graham's conjecture has been verified, it is {\it italicized} if Graham's conjecture has not been verified, and it is \underline{underlined} if the upper bound is (or would be) best possible.  (For example, we know that $\pi(L \Osq L) \geq 64 = \pi(L)^2$ because $|V(L \Osq L)| = 64$.) 
	\item The bottom number reports the run-time of the algorithm.
\end{itemize}
It is worth noting that Graham's conjecture suggests only an upper bound for the pebbling number of the graph product. In some cases, the pebbling number of the graph product may be less than the upper bound proposed by Graham's conjecture. 

Results are organized as follows.  Tables \ref{tab.lemkelemke}, \ref{tab.lemkesmall}, and \ref{tab.lemkebig} display results in which at least one of the base graphs is a Lemke graph.  For most of these cases, Graham's conjecture has not been resolved.

In Tables \ref{tab.smallsmall}, \ref{tab.smallbig} and \ref{tab.bigbig}, all base graphs are non-Lemke graphs.  Paths, cycles, complete graphs, and complete bipartite graphs are all represented.  In most of these test cases, Graham's conjecture has been resolved. In many, the bound supplied by Graham's conjecture is tight. Note the missing data in the case of, $P_{12} \Osq P_{12}$.  The model failed on this instance due to insufficient memory.

Within each set of three tables of results, the tables are further organized by sizes of the non-Lemke base graphs:  ``small'' refers to a graph on 7 or 8 vertices, and ``large'' refers to a graph on 11 or 12 vertices.

%Tables \ref{tab.lemkecomp} and \ref{tab.lemkecomp2} display upper bounds obtained by our model for graph products involving Lemke graphs. 

\begin{table}[h!]
\centering
\noindent {\renewcommand{\arraystretch}{4}}
\begin{tabular}{ | c | c  |  c  |  c | }
\hline
& $L$ & $L_1$ & $L_2$ %& $C_{7}$ & $C_{8}$ & $P_{8}$ & $K_{4,4}$ & $K_{8}$ 
%$T_{8}$ &
%$C_{11}$ &$C_{12}$ &$P_{12}$ &$K_{6,6}$ & $K_{12}$ %& $T_{12}$  
\\ \hline 
$L$ & 
\shortstack[c]{85 \\ \underline{\it (64)} \\ 897.0s } & % L, L 
\shortstack[c]{85 \\ \underline{\it (64)} \\ 374.0s } & % L_1, L 
\shortstack[c]{84 \\ \underline{\it (64)} \\ 1209.9s } % L_2, L 
%\shortstack[c]{108 \\ {\it (88)} \\ 29.0 } & % C_7, L 
%\shortstack[c]{152 \\ \underline{\it (128)} \\ 23.5s }  & % C_8, L 
%\shortstack[c]{1043 \\ \underline{\bf (1024)} \\ 132.9s } & % P_8, L 
%\shortstack[c]{64 \\ \underline{\it (64)} \\ 3.8s } & % K_{4,4}, L 
%\shortstack[c]{64 \\ \underline{\bf (64)} \\ 5.5s }   % K_8, L 
% & % T_8, L 
%\shortstack[c]{? \\ {\it (344)} \\ 9.999s }  & % C_11, L 
%\shortstack[c]{? \\ {\it (512)} \\ 9.999s } & % C_12, L 
%\shortstack[c]{? \\ {\bf (32768)} \\ 9.999s}  & % P_12, L 
%\shortstack[c]{? \\ \underline{\it (96)} \\ 9.999s}  & % K_66, L 
%\shortstack[c]{? \\ \underline{\bf (96)} \\ 9.999s}   % K_12, L 
% % T_12, L 
 \\ \hline 
 
$L_1$ &
& %
\shortstack[c]{84 \\ \underline{\it (64)} \\ 635.5s } & % L_1, L 1
\shortstack[c]{84 \\ \underline{\it (64)} \\ 665.1s } % L_2, L1 
%\shortstack[c]{108 \\ {\it (88)} \\ 9.999s } & % C_7, L 1
%\shortstack[c]{152 \\ \underline{\it (128)} \\ 24.6s }  & % C_8, L 1
%\shortstack[c]{1043 \\ \underline{\it (1024)} \\ 133.0s } & % P_8, L 1 
%\shortstack[c]{64 \\ \underline{\it (64)} \\ 4.0s } & % K_{4,4}, L 1
%\shortstack[c]{64 \\ \underline{\it (64)} \\ 5.2s }  % K_8, L 1
% & % T_8, L 1
%\shortstack[c]{? \\ {\it (344)} \\ 9.999s }  & % C_11, L 1
%\shortstack[c]{? \\ {\it (512)} \\ 9.999s } & % C_12, L 1
%\shortstack[c]{? \\ {\it (32768)} \\ 9.999s}  & % P_12, L 1
%\shortstack[c]{? \\ \underline{\it (96)} \\ 9.999s}  & % K_66, L 1
%\shortstack[c]{? \\ \underline{\it (96)} \\ 9.999s}  % K_12, L 1
% % T_12, L 1
 \\ \hline  
 $L_2$ &
& %
& %
\shortstack[c]{84 \\ \underline{\it (64)} \\ 1183.1s } % L_2, L2
%\shortstack[c]{107 \\ {\it (88)} \\ 10.8 } & % C_7, L2
%\shortstack[c]{150 \\ \underline{\it (128)} \\ 8.8s }  & % C_8, L2
%\shortstack[c]{1041 \\ \underline{\it (1024)} \\ 12.3s } & % P_8, L2
%\shortstack[c]{64 \\ \underline{\it (64)} \\ 3.8s } & % K_{4,4}, L2
%\shortstack[c]{64 \\ \underline{\it (64)} \\ 5.1s }  % K_8, L2
%% & % T_8, L2
%\shortstack[c]{? \\ {\it (344)} \\ 9.999s }  & % C_11, L2
%\shortstack[c]{? \\ {\it (512)} \\ 9.999s } & % C_12, L2
%\shortstack[c]{? \\ {\it (32768)} \\ 9.999s}  & % P_12, L2
%\shortstack[c]{? \\ \underline{\it (96)} \\ 9.999s}  & % K_66, L2
%\shortstack[c]{? \\ \underline{\it (96)} \\ 9.999s}  % K_12, L2
% % T_12, L2  
\\ \hline  
\end{tabular}
\caption{$G$ and $H$ are both Lemke graphs. 
}
\label{tab.lemkelemke}
\end{table}

\begin{table}[h!]
\centering
\noindent {\renewcommand{\arraystretch}{4}}
\begin{tabular}{ | c  |  c| c|c|c|c|}
\hline
& $C_{7}$ & $C_{8}$ & $P_{8}$ & $K_{4,4}$ & $K_{8}$ 
%$T_{8}$ &
%$C_{11}$ &$C_{12}$ &$P_{12}$ &$K_{6,6}$ & $K_{12}$ %& $T_{12}$  
\\ \hline 
$L$ & 
%\shortstack[c]{85 \\ \underline{\it (64)} \\ 897.0s } & % L, L 
%\shortstack[c]{85 \\ \underline{\it (64)} \\ 374.0s } & % L_1, L 
%\shortstack[c]{84 \\ \underline{\it (64)} \\ 1209.9s }& % L_2, L 
\shortstack[c]{108 \\ {\it (88)} \\ 29.0 } & % C_7, L 
\shortstack[c]{152 \\ \underline{\it (128)} \\ 23.5s }  & % C_8, L 
\shortstack[c]{1043 \\ \underline{\bf (1024)} \\ 132.9s } & % P_8, L 
\shortstack[c]{64 \\ \underline{\it (64)} \\ 3.8s } & % K_{4,4}, L 
\shortstack[c]{64 \\ \underline{\bf (64)} \\ 5.5s }   % K_8, L 
% & % T_8, L 
%\shortstack[c]{? \\ {\it (344)} \\ 9.999s }  & % C_11, L 
%\shortstack[c]{? \\ {\it (512)} \\ 9.999s } & % C_12, L 
%\shortstack[c]{? \\ {\bf (32768)} \\ 9.999s}  & % P_12, L 
%\shortstack[c]{? \\ \underline{\it (96)} \\ 9.999s}  & % K_66, L 
%\shortstack[c]{? \\ \underline{\bf (96)} \\ 9.999s}   % K_12, L 
% % T_12, L 
 \\ \hline 
 
$L_1$ &
%& %
%\shortstack[c]{84 \\ \underline{\it (64)} \\ 635.5s } & % L_1, L 1
%\shortstack[c]{84 \\ \underline{\it (64)} \\ 665.1s }& % L_2, L1 
\shortstack[c]{108 \\ {\it (88)} \\ 9.999s } & % C_7, L 1
\shortstack[c]{152 \\ \underline{\it (128)} \\ 24.6s }  & % C_8, L 1
\shortstack[c]{1043 \\ \underline{\it (1024)} \\ 133.0s } & % P_8, L 1 
\shortstack[c]{64 \\ \underline{\it (64)} \\ 4.0s } & % K_{4,4}, L 1
\shortstack[c]{64 \\ \underline{\it (64)} \\ 5.2s }  % K_8, L 1
% & % T_8, L 1
%\shortstack[c]{? \\ {\it (344)} \\ 9.999s }  & % C_11, L 1
%\shortstack[c]{? \\ {\it (512)} \\ 9.999s } & % C_12, L 1
%\shortstack[c]{? \\ {\it (32768)} \\ 9.999s}  & % P_12, L 1
%\shortstack[c]{? \\ \underline{\it (96)} \\ 9.999s}  & % K_66, L 1
%\shortstack[c]{? \\ \underline{\it (96)} \\ 9.999s}  % K_12, L 1
% % T_12, L 1
 \\ \hline  
 $L_2$ &
%& %
%& %
%\shortstack[c]{84 \\ \underline{\it (64)} \\ 1183.1s }& % L_2, L2
\shortstack[c]{107 \\ {\it (88)} \\ 10.8 } & % C_7, L2
\shortstack[c]{150 \\ \underline{\it (128)} \\ 8.8s }  & % C_8, L2
\shortstack[c]{1041 \\ \underline{\it (1024)} \\ 12.3s } & % P_8, L2
\shortstack[c]{64 \\ \underline{\it (64)} \\ 3.8s } & % K_{4,4}, L2
\shortstack[c]{64 \\ \underline{\it (64)} \\ 5.1s }  % K_8, L2
%% & % T_8, L2
%\shortstack[c]{? \\ {\it (344)} \\ 9.999s }  & % C_11, L2
%\shortstack[c]{? \\ {\it (512)} \\ 9.999s } & % C_12, L2
%\shortstack[c]{? \\ {\it (32768)} \\ 9.999s}  & % P_12, L2
%\shortstack[c]{? \\ \underline{\it (96)} \\ 9.999s}  & % K_66, L2
%\shortstack[c]{? \\ \underline{\it (96)} \\ 9.999s}  % K_12, L2
% % T_12, L2  
\\ \hline  
\end{tabular}
\caption{$G :=$ Lemke; $H :=$ ``small'' non-Lemke.
}
\label{tab.lemkesmall}
\end{table}

\begin{table}[h!]
\centering
\noindent {\renewcommand{\arraystretch}{4}}
\begin{tabular}{ | c | c  |  c  |  c  |  c| c|c|c|c|c|c|c|c|c|c| c|}
\hline & $C_{11}$ &$C_{12}$ &$P_{12}$ &$K_{6,6}$ & $K_{12}$ %& $T_{12}$  
\\ \hline 
$L$ & 
\shortstack[c]{383 \\ {\it (344)} \\ 15.9s }  & % C_11, L 
\shortstack[c]{553 \\ {\it (512)} \\ 20.0s } & % C_12, L 
\shortstack[c]{16415 \\ {\bf (16384)} \\ 20.5s}  & % P_12, L 
\shortstack[c]{96 \\ \underline{\it (96)} \\ 46.7s}  & % K_66, L 
\shortstack[c]{96 \\ \underline{\bf (96)} \\ 4145.7s}   % K_12, L 
 % T_12, L 
 \\ \hline 
 
$L_1$ &
\shortstack[c]{389 \\ {\it (344)} \\ 17.1s }  & % C_11, L 1
\shortstack[c]{554 \\ {\it (512)} \\ 30.5s } & % C_12, L 1
\shortstack[c]{16416 \\ {\it (16384)} \\ 14.7s}  & % P_12, L 1
\shortstack[c]{96 \\ \underline{\it (96)} \\ 48.1s}  & % K_66, L 1
\shortstack[c]{96 \\ \underline{\it (96)} \\ 234.8s}  % K_12, L 1
 % T_12, L 1
 \\ \hline  
 $L_2$ &
\shortstack[c]{379 \\ {\it (344)} \\ 5.6s }  & % C_11, L2
\shortstack[c]{548 \\ {\it (512)} \\ 25.1s } & % C_12, L2
\shortstack[c]{16411 \\ {\it (16384)} \\ 13.0s}  & % P_12, L2
\shortstack[c]{96 \\ \underline{\it (96)} \\ 49.4s}  & % K_66, L2
\shortstack[c]{96 \\ \underline{\it (96)} \\ 247.1s}  % K_12, L2
 % T_12, L2  
\\ \hline  
\end{tabular}
\caption{$G :=$ Lemke; $H :=$ ``large'' non-Lemke.
}
\label{tab.lemkebig}
\end{table}

\begin{table}[h!]
\centering
\noindent {\renewcommand{\arraystretch}{4}}\begin{tabular}{ | c |  c| c|c|c|c|c|c|c|c|c|c| c|c|}
\hline
\renewcommand{\arraystretch}{1.3}
& $C_{7}$ & $C_{8}$ & $P_{8}$ & $K_{4,4}$ & $K_{8}$ %&$T_{8}$ &$C_{11}$ &$C_{12}$ &$P_{12}$ &$K_{6,6}$ & $K_{12}$ & $T_{12}$  
\\ \hline 
$C_{7}$ &
\shortstack[c]{140 \\ {\bf (121)} \\ 3.5s }  % C_7, C_7 
&&&&%&&&&&&
\\ \hline
$C_{8}$ &
\shortstack[c]{196 \\ {\bf (176)} \\ 4.5s } & % C_8, C_7
\shortstack[c]{278 \\ \underline{\bf (256)} \\ 12.1s }  % C_8, C_8
&&&%&&&&&&
\\ \hline
$P_{8}$ &
\shortstack[c]{1188 \\ {\bf (1408)} \\ 16.1s } & % P_8, C_7
\shortstack[c]{2063 \\ \underline{\bf (2048)} \\ 23.8s }& % P_8, C_8
\shortstack[c]{16399 \\ \underline{\bf (16384)} \\ 728.2s }   % P_8, P_8
&&%&&&&&&
\\ \hline
$K_{4,4}$ &
\shortstack[c]{76 \\ {\bf (88)} \\ 3.3s } & % K_{4,4}, C_7
\shortstack[c]{104 \\ {\bf (128)} \\ 3.0s } & % K_{4,4}, C_8
\shortstack[c]{562 \\ {\bf (1024)} \\ 16.8s } & % K_{4,4}, P_8
\shortstack[c]{64 \\ \underline{\bf (64)} \\ 5.6s }  % K_{4,4}, K_{4,4}
&%&&&&&&
\\ \hline
$K_{8}$ &
\shortstack[c]{67 \\ {\bf (88)} \\ 5.3s } & % K_8, C_7
\shortstack[c]{86 \\ {\bf (128)} \\ 5.2s } & % K_8, C_8
\shortstack[c]{311 \\ {\bf (1024)} \\ 28.0 s } & % K_8, P_8
\shortstack[c]{64 \\ \underline{\bf (64)} \\ 5.7s }  & % K_8, K_{4,4}
\shortstack[c]{64 \\ \underline{\bf (64)} \\ 7.0s }   % K_8, K_8
%&&&&&&
\\ \hline

\end{tabular}
\caption{$G$ and $H$ are both ``small'' non-Lemke graphs.
}
\label{tab.smallsmall}
\end{table}

\begin{table}[h!]
\centering \renewcommand{\arraystretch}{3}
\begin{tabular}{ | c |  c| c|c|c|c|c|c|c|c|c|c| c|c|}
\hline
& $C_{11}$ &$C_{12}$ &$P_{12}$ &$K_{6,6}$ & $K_{12}$  % $T_{12}$ 
 \\ \hline 
$C_{7}$ &
\shortstack[c]{491 \\ {\bf (473)} \\ 3.8s }  & % C_11, C_7 
\shortstack[c]{712 \\ {\bf (704)} \\ 4.8s }  & % C_12, C_7
\shortstack[c]{16636 \\ {\bf (22528)} \\ 15.9s}  & % P_12, C_7
\shortstack[c]{106 \\ {\bf (132)} \\ 49.0s}  & % K_66, C_7
\shortstack[c]{95 \\ {\bf (132)} \\ 49.5s}  % K_12, C_7 
\\ \hline $C_{8}$ &
\shortstack[c]{721 \\ {\bf (688)} \\ 6.2s }  & % C_11, C_8
\shortstack[c]{1060 \\ \underline{\bf (1024)} \\ 16.7s }  & % C_12, C_8
\shortstack[c]{32797 \\ \underline{\bf (32768)} \\ 332.9s}  & % P_12, C_8
\shortstack[c]{140 \\ {\bf (192)} \\ 73.7s}  & % K_66, C_8
\shortstack[c]{118 \\ {\bf (192)} \\ 304.1s}  % K_12, C_8
\\ \hline $P_{8}$ &
\shortstack[c]{4975 \\ {\bf (5504)} \\ 253.8s }  & % C_11, P_8
\shortstack[c]{8217 \\ \underline{\bf (8192)} \\ 365.3s }  & % C_12, P_8
\shortstack[c]{262164 \\ \underline{\bf (262144)} \\ 144.9s}  & % P_12, P_8
\shortstack[c]{611 \\ {\bf (1536)} \\ 455.8s}  & % K_66, P_8
\shortstack[c]{343 \\ {\bf (1536)} \\ 437.6s} % K_12, P_8
\\ \hline $K_{4,4}$ &
\shortstack[c]{240 \\ {\bf (344)} \\ 5.8s } & % C_11, K_{4,4}
\shortstack[c]{331 \\ {\bf (512)} \\ 5.1s } & % C_12, K_{4,4}
\shortstack[c]{8267 \\ {\bf (16384)} \\ 123.9s} & % P_12, K_{4,4}
\shortstack[c]{96 \\ \underline{\bf (96)} \\ 48.9s}  & % K_66, K_{4,4}
\shortstack[c]{96 \\ \underline{\bf (96)} \\ 45.0s}  % K_12, K_{4,4}
\\ \hline $K_{8}$ &
\shortstack[c]{162 \\ {\bf (344)} \\ 4.9s } & % C_11, K_8
\shortstack[c]{211 \\ {\bf (512)} \\ 5.2s } & % C_12, K_8
\shortstack[c]{4179 \\ {\bf (16384)} \\ 139.5s} & % P_12, K_8
\shortstack[c]{96 \\ \underline{\bf (96)} \\ 44.6s} & % K_66, K_8
\shortstack[c]{96 \\ \underline{\bf (96)}  \\ 47.3s}   % K_12, K_8
\\ \hline % T_12, T_12
\end{tabular}
\caption{$G :=$ ``small'' non-Lemke; $H :=$ ``large'' non-Lemke.
}
\label{tab.smallbig}

\end{table}

\begin{table}[h!]
\centering \renewcommand{\arraystretch}{3}
\begin{tabular}{ | c |  c| c|c|c|c|c|c|c|c|c|c| c|c|}
\hline
& $C_{11}$ &$C_{12}$ &$P_{12}$ &$K_{6,6}$ & $K_{12}$  % $T_{12}$ 
 \\ \hline 
$C_{11}$ &
\shortstack[c]{1908 \\ {\bf (1849)} \\ 56.0s }  & % C_11, C_11
&&& 
\\ \hline $C_{12}$ &
\shortstack[c]{2804 \\ {\bf (2752)} \\ 66.3s } & % C_12, C_11 4096
\shortstack[c]{4158 \\ \underline{\bf (4096)} \\ 76.4s}    % C_12, C_12
&&&\\ \hline
$P_{12}$ &
\shortstack[c]{66873 \\ {\bf (88064)} \\ 72.8s} & % P_12, C_11
\shortstack[c]{131110 \\ \underline{\bf (131072)} \\ 183.6s} & % P_12, C_12
 \shortstack[c]{-- \\ \underline{\bf (16777216)} \\ --} % P_12, P_12
&& \\ \hline
$K_{6,6}$ &
\shortstack[c]{306 \\ {\bf (516)} \\ 83.3s}  & % K_66, C_11
\shortstack[c]{406 \\ {\bf (768)} \\ 77.7s}  & % K_66, C_12
\shortstack[c]{8342 \\ {\bf (24576)} \\ 1621.8s}  & % K_66, P_12
\shortstack[c]{144 \\ \underline{\bf (144)} \\ 94.1s}   % K_66, K_66
&  \\ \hline
$K_{12}$ &
\shortstack[c]{206 \\ {\bf (516)}  \\ 271.9s}  & % K_12, C_11
\shortstack[c]{259 \\ {\bf (768)} \\ 282.5s} & % K_12, C_12
\shortstack[c]{4227 \\ {\bf (24576)} \\ 7280.6s}  & % K_12, P_12
\shortstack[c]{144 \\ \underline{\bf (144)} \\ 87.9s}  & % K_12, K_66
\shortstack[c]{144 \\ \underline{\bf (144)} \\ 487.7s} % K_12, K_12
\\ \hline % T_12, T_12
\end{tabular}
\caption{$G$ and $H$ are both ``large'' non-Lemke graphs.
}
\label{tab.bigbig}

\end{table}

\subsection{Discussion of results} \label{sec.disc}

Addressing first the primary motivation for our work, we found that $\pi(L \Osq L) \le 85$, which is an improvement over the previous upper bound of  $\pi(L \Osq L) \le 91$ discovered by an earlier version of our model (\cite{COCOApebbling}), and a vast improvement over the best known bound prior to our model, $\pi(L \Osq L) \le 108$ (\cite{linopt2011}).  While $85$ is still more than 32\% above the conjectured bound of $\pi(L \Osq L) = 64$, one must keep in mind that $L \Osq L$ is also a proposed counterexample to Graham's conjecture \cite{gross2013handbook}. Hence, the bound $\pi(L \Osq L) \le 85$ may be more accurate than it otherwise seems. Interestingly, the model attains nearly the same bounds (84 or 85) on products involving the three minimal Lemke graphs.

The model found the exact pebbling number for every test graph made up of a Lemke graph composed with either a complete graph, or a complete bipartite graph.  In each such case, the bound attained by the model is the best possible bound, so it must be $\pi(G \Osq H)$.  In \cite{gao2017lemke}, Gao and Yin proved that $\pi(L \Osq K_n) = \pi(L)\pi(K_n) = 8n$, which matches the bounds provided by our model for $\pi(L \Osq K_8)$ and $\pi(L \Osq K_{12})$.  The rest of the results listed in Proposition \ref{Lcomplete} are new.

\begin{prop} \label{Lcomplete} Let $G \in \{L, L_1, L_2\}$.  If $H \in \{K_8, K_{4,4}\}$, then \[ \pi(G \Osq H) = 64.\]  If $H \in \{K_{12}, K_{6,6}\}$, then \[\pi(G \Osq H) = 96. \]
\end{prop}

\noindent Proposition \ref{Lcomplete} settles Graham's conjecture for 10 new Cartesian product graphs.
 
\begin{cor} Graham's conjecture holds for $G \Osq H$, where $G \in \{L, L_1, L_2\}$ and $H \in \{K_8,K_{12}, K_{4,4}, K_{6,6}\}$
\end{cor}

\noindent These results provide compelling evidence that Graham's conjecture holds for the product of any minimal Lemke graph with any complete graph or complete bipartite graph.
%The following hold:
%\begin{align*}
%\pi(L \Osq K_8) = 64 & \quad &
%\pi(L_1 \Osq K_8) = 64 \\
%\pi(L_2 \Osq K_8) = 64 & \quad &
%\pi(L \Osq K_{4,4}) = 64 \\
%\pi(L_1 \Osq K_{4,4}) = 64 & \quad &
%\pi(L_2 \Osq K_{4,4}) = 64 \\
%\pi(L \Osq K_{12}) = 96 & \quad &
%\pi(L_1 \Osq K_{12}) = 96 \\
%\pi(L_2 \Osq K_{12}) = 96 & \quad &
%\pi(L \Osq K_{6,6}) = 96 \\
%\pi(L_1 \Osq K_{6,6}) = 96 & \quad &
%\pi(L_2 \Osq K_{6,6}) = 96 \\
%\end{align*}

%The results $\pi(L \Osq K_8) = 64$ and $\pi(L \Osq K_{12}) = 96$ were proved by . All of the exact computations involving $L_1$ and $L_2$ are new. Since $L, L_1$, and $L_2$ are all minimal Lemke graphs, this result also implies that $\pi(L^* \Osq K_8) = 64$ and $\pi(L^* \Osq K_{12}) = 96$ for {\it any} Lemke graph on 8 vertices.

Our results capture the phenomenon that $\pi(G \Osq H)$ may be substantially smaller that $\pi(G)\pi(H)$. While $\pi(G \Osq H) = \pi(G)\pi(H)$ for many well-understood cases (such as when $G$ and $H$ are both paths or they are both complete graphs), it is not trivial to show $\pi(G \Osq H)$ could be substantially smaller than $\pi(G)\pi(H)$. Our computations provide several interesting concrete examples of this occurrence, including the products of complete graphs and paths. While this phenomenon is not surprising, per se, it is surprising how small $\pi(G \Osq H)$ could be compared to $\pi(G)\pi(H)$. Indeed, the spread of bounds on $\pi(G \Osq H)$ attained by the model reinforces the difficulty of Graham's conjecture.

\section{Conclusion} \label{sec.conclusion}

We describe an IP-based algorithm that finds an upper bound on $\pi(G \Osq H)$, given base graphs $G$ and $H$, and provide the results of applying the algorithm to a variety of cross-product graphs. Our main approach is to leverage the underlying symmetry of graph products by encoding partial-pebbling strategies in our constraints.  Modeling at the level of slices within a frame graph, rather than at the indidual vertex level, allows us to eliminate large classes of pebbling strategies with each constraint.   

The IP-based algorithm finds an improved upper bound on $\pi(L \Osq L)$, and similar upper bounds for pebbling numbers of products involving the other two minimal Lemke graphs.  Yet, the algorithm did not reach the conjectured bound of $\pi(L)^2$ for $\pi(L \Osq L)$. One may interpret this as evidence (however weak) that Graham's conjecture may be false; however, it is likely more indicative of the extreme difficulty of the problem. 

In other cases, such as for  $\pi(L_1 \Osq K_{4,4})$, our algorithm attains exact values. Our IP-based algorithm adapts well to many different base graphs, and in several cases we bound the pebbling number close to the exact value, or in other cases, below the upper bound suggested by Graham's conjecture.

Our computational results inspire other theoretical questions.  For example, we demonstrate that some product graphs have pebbling numbers that fall substantially below the bound suggested by Graham's conjecture. This leads to the question: ``Is there a Graham-like {\it lower bound} for the pebbling number of two graphs?'' Further, while most of these product graphs involve a sparse graph (e.g., a path) and a dense graph (e.g., a complete graph), we were unable to find any examples where $\pi(G \Osq G)$ is substantially less than $\pi(G)^2$. It would be interesting to find such an example.

With regard to improving our model, it seems unlikely that our approach will lead to an exact bound on $\pi(L \Osq L)$.  However, the low computation times for many test cases suggest that marginal gains are within reach by including more pebbling constraints.   There are two likely avenues for finding constraints that will improve the current IP.  The first is to leverage results that are know about products of special classes of graphs. We could use these to bound the number of pebbles placed on product subgraphs of $G \Osq H$ that contain the root node. The second is to encode ``the next level'' partial-pebbling strategies. This would involve studying the bounding configuration supplied by the model for $L \Osq L$, for example, and using it to identify pebbling strategies that are not yet captured by constraints.

% ----------------------------------------------------------------
\section{References}

\bibliographystyle{siam}
\bibliography{../pebble.bib}

\end{document}